\documentclass[a4paper, 10pt]{article}
\usepackage{amsmath,amsthm,amssymb}
\usepackage[T1]{fontenc}
\usepackage[margin=3.1cm]{geometry}
\usepackage[numbered]{bookmark}
\usepackage{changepage} 
\usepackage[utf8]{inputenc}
\usepackage{graphics, graphicx, caption}
\usepackage{amsmath, amsfonts, amsthm, amssymb}
\usepackage{mathrsfs, mathtools, mathabx}
\usepackage{verbatim} 
\usepackage{enumerate}
\usepackage{xcolor}
\usepackage{cite}
\newcommand{\vep}{\varepsilon}

\newcommand{\RV}{\mathcal{RV}}
\newtheorem{theorem}{Theorem}[section]
\newtheorem{definition}[theorem]{Definition}
\newtheorem{lemma}[theorem]{Lemma}
\newtheorem{proposition}[theorem]{Proposition}

\newtheorem{remark}{Remark}

\DeclareMathOperator{\id}{id}

\newcommand{\Z}{\mathbb{Z}} 
\newcommand{\N}{\mathbb{N}}
\newcommand{\Q}{\mathbb{Q}}
\newcommand{\R}{\mathbb{R}}

\newcommand{\la}{\lambda}
\newcommand{\A}{\mathcal{A}}

\newcommand{\Function}[5]{\begin{array}{cccc} #1 :& #2 & \rightarrow & #3 \\ & #4 & \mapsto & #5 \end{array}}

\date{}
\begin{document}
\title{Ergodicity of explicit logarithmic cocycles over IETs}
\author{P. Berk, F. Trujillo, C. Ulcigrai}
\maketitle

\abstract{We prove ergodicity in a class of skew-product extensions of interval exchange transformations given by cocycles with logarithmic singularities. This, in particular, provides explicit examples of ergodic $\mathbb{R}$-extensions of minimal locally Hamiltonian flows with non-degenerate saddles in genus two. More generally, given any symmetric irreducible permutation, we show that for almost every choice of lengths, the skew-product built over the associated IET using a cocycle with symmetric logarithmic singularities that is \emph{odd} when restricted to each of the exchanged intervals is ergodic.}

\section{Introduction} 
This paper gives a contribution to the study of \emph{skew-products} over interval exchange transformations, an infinite measure preserving setting in which fundamental questions such as ergodicity are still wide open. 

Let $I:=[0,1]$ denote the unit interval and $T: I\to I$ be an \emph{interval exchange transformation} (or IET for short), i.e., a piecewise isometry of the interval with finitely many discontinuities (see \S~\ref{sec:IETs} for the formal definition). 

Given a real-valued measurable function $\varphi : I \to \mathbb{R}$, the $\mathbb{R}$-extension of $T$ given by $\varphi $ is the skew-product automorphism $T_\varphi $ of the strip $I\times \R$ given by
\begin{equation}\label{IETskew}
T_\varphi (x,y)= (T(x), y+ \varphi (x)), \qquad x\in I, \quad y\in \R.
\end{equation}
One can then see that $T_\varphi$ preserves a natural infinite invariant measure, namely, the restriction $\lambda$ to the strip $I\times \R$ of the $2$-dimensional Lebesgue measure. Iterates of $T_\varphi $ have the form 
$$T^n_\varphi (x,y)= (T^n(x), y+ S_n \varphi (x)),$$
where $S_n \varphi $ denotes the \emph{Birkhoff sums} of $\varphi $ over $T$, which are given by 
\begin{equation}\label{def:BS}
S_n \varphi := 
\begin{cases} \sum_{j=0}^{n-1}\varphi \circ T^j , & \text{if}\ n>0,\\
0 , & \text{if}\ n=0,\\
\sum_{j=n}^{-1}\varphi \circ T^{-j} , & \text{if}\ n<0.\\
\end{cases}
\end{equation}
 Since Birkhoff sums define a $\mathbb{R}$-\emph{cocycle} over $T$ (i.e., $S_{n+m}\varphi =S_n \varphi +S_m \varphi \circ T^n$ for all $n,m\in\mathbb{Z}$) 
the function $\varphi $ is also called a \emph{cocycle}. It is well known that $T_\varphi $ is recurrent if and only if $\int_I \varphi (x) dx=0$ {(see~\cite{atkinson}).} { Recall that $T_\varphi$ is called \emph{recurrent} if for every $A \subseteq I$ of positive Lebesgue measure and any $\epsilon > 0$ there exists $n \in \Z$ such that $T_\varphi^n(A \times \{0\}) \cap A \times (-\epsilon, \epsilon) \neq \emptyset$.} 

Skew-products of the form \eqref{IETskew} where $T: I \to I$ is a \emph{rotation} (i.e., the map $T(x)=x+\alpha \mod 1$) are one of the most studied examples in infinite ergodic theory (see for example the monographs \cite{aaronson_introduction_1997, schmidt_cocycles_1977} { and the references therein}). In particular, when the cocycle $\varphi $ is of the form $\varphi =\chi_{[0, \beta]}-\beta$ {(where $\chi_{[a, b]}$ denotes the characteristic function of the interval $[a,b]$ and $0 < \beta < 1$)} there is a long history of results on it, starting from ergodicity (see for example \cite{schmidt_cylinder_1978,conze_ergodicite_1976,aaronson_visitors_1982,oren_ergodicity_1983,avila_visits_2015,aaronson2016discrepancy}). Interval exchange transformations are a natural generalization of rotations. However, the problem of ergodicity of skew-product extensions over IETs is still {widely} open despite being actively researched; see, for example, \cite{Hu-We, Co-Fr, Fr-Ul_Ehr, Fr-Hu, Ho, Ra-Tr1, Ra-Tr2, Ch-R} for some results in particular settings. In particular, some specific results for skew-products over IETs with $d>2$, when the cocycle $\varphi $ is piecewise constant or piecewise absolutely continuous, were proved, for example, in \cite{Co-Fr, Fr-Ul_Ehr, marmi_cohomological_2005}.

\subsection{Ergodicity of cocycles with logarithmic singularities}\label{sec:intro}

In this paper, we consider the special case when the cocycle $\varphi: I\to$ {$\mathbb{R}\cup \{+ \infty, -\infty\}$} has \emph{logarithmic singularities}, {namely, if $\{ u^t_i, 1\leq i\leq d-1\} $ denote the discontinuities $T$,} it blows up at a (nonempty subset) of { $\{u^t_i, 1\leq i\leq d-1\} \cup \{u^t_0:=0, u^t_d:=1\}$ as $C_{i}^\pm |\log (x-u^t_i)|$ when $x\to (u^t_i)^{\pm}$}, { for some constants $C_i^\pm \in \R$ (not all equal to 0)} (see Definition~\ref{def:logsing}). 
One of the reasons for the interest in this type of cocycle is that skew-products over IETs under such cocycles arise naturally when taking Poincaré sections of extensions of locally Hamiltonian flows on surfaces of genus $g \geq 1$ (see \S~\ref{sec:reduction} and also the introduction of \cite{Fr-Ul_BS} and the references therein), a class of flows whose study goes back to the work of Poincaré \cite{poincare_memoire_1881} (see also the work of Krygin \cite{Kr}). 

When $d=2$, i.e., $T=R_\alpha$ is a rotation, ergodicity of skew-products as in \eqref{IETskew} for cocycles with logarithmic singularities was previously studied for cocycles with one {\emph{symmetric}} logarithmic singularity {(see Definition \ref{def:symlog})}, such as 
{$$\varphi (x)=C\log(x)+ C\log (1-x) + 2C,$$ }
 by Frączek and Lema{\' n}czyk in \cite{Fr-Le}, {and for cocycles} with one \emph{asymmetric} logarithmic singularity {(see Definition \ref{def:symlog})}, {such as}
\begin{equation}\label{1asymm} 
{\varphi (x)=C \log (1-x) + C,}
\end{equation} 
by Fayad and Lema\'nczyk in \cite{FL:ont} {(notice that the constants $2C$ and $C$ are added so that the cocycles have zero mean)}. In both cases, it was proved that $T_\varphi $ is ergodic for almost every rotation number. 

Cocycles with logarithmic singularities over IETs were recently investigated in \cite{fraczek_ergodic_2012, Fr-Ul_BS}. This paper considers a special class of cocycles with logarithmic singularities over IETs in which one can exhibit explicit examples of ergodic extensions. We comment below on the relation between this result and the existing ones in the literature. More precisely, we consider cocycles with non-trivial logarithmic singularities (see Definition \ref{def:logsing}), which are \emph{odd} when restricted to each of the intervals $I_i:=(a_i,b_i)$ exchanged by the IET, i.e., for $i=1,\dots, d$, the cocycle $\varphi $ satisfies 
$$\varphi ( a_{i}+ b_{i} - x)=-{\varphi (x)}, \qquad \text{for\, all}\ x\in (a_i,b_{i}).$$ 
See Definition \ref{def:odd} and \S~\ref{sec:logsing} for more details. Furthermore, we consider IETs whose permutations reverse the order of the intervals, so-called \emph{symmetric} permutations (see \S~\ref{sec:IETs}). 

\begin{theorem}[Ergodicity of odd cocycles with logarithmic singularities]
\label{thm:reductiontoIETs}
For any $d\geq 3$, for almost every IET $T:I \to I$ with a symmetric permutation on $d$ symbols, and for any cocycle $\varphi$ {of class $C^2$} which:
\begin{enumerate}
\item has non-trivial symmetric logarithmic singularities (see \S~\ref{sec:logsing}) at the discontinuities of $T$ (in the sense of Definition~\ref{def:logsing});
\item is \emph{odd} on each continuity interval of $T$ (in the sense of Definition~\ref{def:odd}), 
\end{enumerate} 
the associated skew-product $T_\varphi $ is ergodic { with respect to the Lebesgue measure on $I \times \R$}. 
\end{theorem}
{ We remark that the assumption $d \ge 3$ is necessary, since for $d=2$, for any rotation $R_{\alpha}$ by $\alpha\in\R\setminus\Q$, 
 a cocycle $\varphi : I\to\R$ satisfying the assumptions of the theorem} is necessarily of the form 
\[ 
{ \varphi (x)=C\big(\log x +\log(1-x)- \log(x + \alpha\mod 1) - \log(1- \alpha - x \mod 1)\big)},
\]
{which is} a \emph{coboundary} with respect to $R_\alpha$, i.e., it can be 
expressed in the form $\psi - \psi \circ R_\alpha$ for some measurable function 
$\psi: I \to \R$ (in this case, $\psi = C(\log x +\log(1-x))$, and thus the 
associated skew-product is not ergodic (see \cite[Chapter 
8]{aaronson_introduction_1997}).

For cocycles with logarithmic singularities over IETs, the only results about ergodicity were proved by Frączek and the last author (see \cite{Fr-Ul_BS}, which generalizes previous work in \cite{fraczek_ergodic_2012} in the particular case of periodic-type IETs). In \cite{Fr-Ul_BS}, the \emph{existence} of ergodic skew-products over IETs given by cocycles with symmetric logarithmic singularities is proved for any $d\geq 2$ and any combinatorics, but the class of cocycles considered is not explicit. In particular, it is shown that for almost every IET and every cocycle $\varphi $ with symmetric logarithmic singularities, there exists a piecewise constant function $\chi$, constant on the intervals exchanged by $T$, called \emph{correction}, such that the skew-product $T_{\varphi -\chi}$ given by the corrected cocycle $\varphi -\chi$ is ergodic. While this provides a large class of cocycles with symmetric logarithmic singularities whose $\R$-extension is ergodic (indeed, corrected cocycles form a subspace of finite codimension), it is not an \emph{explicit} class (since it is not clear how to check whether a cocycle is indeed \emph{corrected}). On the other hand, the assumption we make on the roof, namely being odd with respect to the involution, is explicit and checkable. 


As an application of Theorem~\ref{thm:reductiontoIETs}, we build ergodic extensions of locally Hamiltonian flows in genus two (see \S~\ref{sc:loc_ham}). We state {the precise} result in the next section (see Theorem~\ref{thm:ergodic_extension}).

\subsection{Ergodicity of locally Hamiltonian flows extensions}
\label{sc:loc_ham}
Let $\phi_\R = (\phi_t)_{t \in \R}$ be a smooth area-preserving flow (or equivalently, a locally Hamiltonian flow, see for example \cite{Fr-Ul_BS}) on a compact closed surface $M$. Given a continuous function $f:M \to \R$, the associated extension flow on $M \times \R$ is denoted by $(\phi_t^f)_{t \in \R}$ and given by
\begin{equation}
\label{eq:extension_flow}
 \phi_t^f(x, a) = \left(\phi_t(x), a + \int_0^t f(\phi_s(x))ds\right),
\end{equation}
for any $x \in M$, and any $a, t \in \R.$ Notice that the flow {$\phi^f_\R$} has a skew-product form, i.e., it \emph{projects} on the $M$ coordinate to the flow $\phi_\R$, and that the motion in the $\R$ fiber is determined by the oscillations of the \emph{ergodic integrals} of $f$ along $\phi_\R$. These flows arise {in the work of Poincaré \cite{poincare_sur_1885}} by coupling $\phi_\R$ with the solutions of the differential equation on $\R$
\[
\frac{\mathrm{d}\,\! y}{\mathrm{d}\,\! t}= f(\psi_t(p)), \qquad y\in \R,\ \quad t\in \R.
\]
For surfaces of \emph{genus one}, the existence of ergodic extensions was first discovered by Krygin \cite{Kr} in the case where the flow $\phi_\R$ has no singularities. {
For flows on the torus with singularities, a much-studied class of locally Hamiltonian flows first studied by Arnold in \cite{arnold_topological_1991},} {and, in particular,} for an \emph{Arnold flow} (namely restriction to the minimal component of a locally Hamiltonian flow in genus one with one saddle and one center), ergodicity of extensions was proved for a full measure set of rotation numbers by Fayad and Lema\'nczyk in \cite{FL:ont} (as a corollary of the ergodicity of skew-products over rotations under an asymmetric cocycle $\varphi $ of the form \eqref{1asymm} {(the reduction is also explained in \S~\ref{sec:reduction} below)).}

As a consequence of Theorem~\ref{thm:reductiontoIETs} we can produce explicit examples of ergodic extensions of locally Hamiltonian flows in genus two. Let us recall that minimal locally Hamiltonian flows with only saddles are (singular) time-reparametrizations of linear flows on translation surfaces (an observation which goes back to Zorich, see \cite{Zo:how}, \cite{Rav:mix}). { More precisely, for a given locally Hamiltonian flow $(\phi_t)_{t \in \R}$, there exists a translation structure on $M$ and a non-negative function {$\alpha:M\to \R_{\geq 0}$, which is as regular as the vector field associated to $(\phi_t)_{t \in \R}$ and} vanishes at saddle points, such that if $X_v$ is the vertical vector field with respect to the translation structure, $(\phi_t)_{t \in \R}$ is generated by the rescaled vector field 
$\alpha X_v$.

We will say that {$\alpha$ is the \emph{generator} of} the time-change $(\phi_t)_{t \in \R}$. Let us also recall that every translation surface in genus two admits a \emph{hyperelliptic involution}, i.e., a continuous map $\mathcal{I}: M\to M$ such that $\mathcal{I} \neq \id_M$, $\mathcal{I}^2$ { is equal to the identity,} {and verifies 
\begin{equation}
\label{eq:vertical_involution} 
v_{{-t}}(\mathcal I(x)) = \mathcal I (v_{t}(x)),
\end{equation}
for any $x \in M$ and any $t \in \R$ for which the flows above are well-defined.
} We will refer to it as the \emph{underlying hyperelliptic involution}.

\smallskip
We will prove the following:

\begin{theorem}
\label{thm:ergodic_extension}
For a typical Hamiltonian flow $(\phi_t)_{t \in \R}$ { of class $C^2$} on a surface $M$ of genus $2$ with two isomorphic saddles, the extension $(\phi^f_t)_{t \in \R}$ given by a {$C^2$} function $f: M\to \mathbb{R}$ is ergodic if 
the function $\alpha f$, where $\alpha$ is the generator of the time-change which gives $(\phi_t)_{t \in \R}$, is \emph{odd} with respect to the underlying hyperelliptic involution $\mathcal{I}$ on $M$, i.e.,
\begin{equation}
\label{eq:odd_observable}
 \alpha f \circ \mathcal{I} = - \alpha f.
\end{equation}
\end{theorem}}
{ The notion of \emph{typical} is measure-theoretical and refers to the measure class induced by the Katok fundamental class (see e.g.~\cite{Ul:ICM}). } 
In \cite{Fr-Ul_BS}, the existence of ergodic extensions of typical minimal locally Hamiltonian flows with non-degenerate saddles was proved in any genus by Frączek and the third author (extending the previous work \cite{fraczek_ergodic_2012} by the same authors treating only a measure zero class of such flows) under the assumption that the observable $f$ belongs to a finite codimension (but still infinite dimensional) subspace (which corresponds to the subspace of \emph{corrected} cocycles mentioned in the previous section \S~\ref{sec:intro}). { The nature of this class, though, is mysterious, and it is unclear how to effectively check that a given cocycle is indeed \emph{corrected} and hence induces an ergodic extension. On the other hand, {as explained below}, Theorem~\ref{thm:ergodic_extension} allows defining an infinite dimensional space of functions that induce explicit examples of ergodic extensions.} 


In \S~\ref{sec:reduction}, we show how Theorem~\ref{thm:ergodic_extension} can be reduced to an application of Theorem~\ref{thm:reductiontoIETs}. Indeed, we will show that, by taking a suitable Poincaré section, the study of the ergodicity of extension flows $(\phi_t)_{t \in \R}$ in the statement of Theorem~\ref{thm:ergodic_extension} can be reduced to the study of ergodicity of skew-products of the form \eqref{eq:skew-product}. 
{ We will show more generally that 
there exists a class $\mathcal{F}_{odd}\subseteq \mathcal{C}^2(M)$ of observables (that we call ``\emph{odd}'') such that for any $f:M\to \mathbb{R}$ in $\mathcal{F}_{odd}$ 
the extension $(\phi^f_t)_{t \in \R}$ is ergodic (see Remark~\ref{rk:odd_observables}) and that this space $\mathcal{F}_{odd}$ contains the infinite-dimensional space of functions $f: M\to \mathbb{R}$ which are odd with respect to the underlying hyperelliptic involution in the sense of \eqref{eq:odd_observable}. }

\subsection{Outline of the proof}
Let us conclude the introduction by presenting some of the ideas in the proof of Theorem~\ref{thm:reductiontoIETs}. To show ergodicity, we make use of a criterion (stated here below as Proposition~\ref{prop:ergodicity_criterion}) for the ergodicity of skew-products introduced and used in \cite{fraczek_ergodic_2012, lemanczyk_ergodic_1996}. Let us first recall the definition of partial rigidity sets. 

\begin{definition}[Partial rigidity { sets}]
\label{def:partial_rigidity}
Given a probability preserving transformation $T:(X,\mu) \to (X, \mu)$, we say that a sequence of Borel subsets $(\Xi_n)_{n \in \N}$ of $X$ is a sequence of \emph{partial rigidity sets} if there exists an increasing sequence of natural numbers $(h_n)_{n \in \N}$, called \emph{rigidity times}, and $\delta>0$ such that, {as $n\to \infty$}: 
\begin{enumerate}[(i)]
 \item \label{cond:towers_measure} $\mu(\Xi_n) \to \delta > 0 $, 
 \item \label{cond:quasi_invariance} $\mu(\Xi_n \Delta T(\Xi_n)) \to 0$, 
 \item \label{cond:partial_rigidity} $\sup_{x \in \Xi_n} d(x, T^{h_n}x) \to 0$.
\end{enumerate}
We then say that $(\Xi_n)_{n \in \N}$ is a sequence of $\delta$-partial rigidity sets along the sequence $(h_n)_{n \in \N}$ of rigidity times.
\end{definition}
\noindent In particular, this definition guarantees that the iterates $T^{h_n}$, when restricted to $\Xi_n$, converge pointwise to the identity. The following criterion shows that ergodicity follows from the existence of partial rigidity sets $(\Xi_n)_{n \in \N}$ and rigidity times $(h_n)_{n \in \N}$ such that the \emph{Birkhoff sums} $(S_{h_n} \varphi (x))_{n\in \N}$ 
are \emph{tight} (Assumption \ref{tightness_assumption} in Proposition \ref{prop:ergodicity_criterion}) when restricted restricted to $(\Xi_n)_{n \in \N}$ while having enough \emph{oscillation} (Assumption \ref{oscillations_assumption} in Proposition \ref{prop:ergodicity_criterion}).

\begin{proposition}[Ergodicity criterion, \cite{fraczek_ergodic_2012, lemanczyk_ergodic_1996}]
\label{prop:ergodicity_criterion}
Let $T:(X,\mu) \to (X, \mu)$ be an automorphism of a compact metric space $X$ preserving an ergodic probability measure $\mu$ and let $\varphi \in L^1(X, \mu)$. Suppose there exist {$\delta$}-partial rigidity sets $(\Xi_n)_{n \in \N}$ {for some $\delta>0$} (in the sense of Definition~\ref{def:partial_rigidity}) and corresponding rigidity times $(h_n)_{n \in \N}$ such that the following two assumptions hold:
\begin{enumerate}
 \item\label{tightness_assumption} \emph{BS Tightness Assumption}: \label{cond:tightness} $\sup_{n \in \N} \int_{\Xi_n} |S_{h_n} \varphi (x) d\mu(x)| < +\infty$.
 \item \label{oscillations_assumption} \emph{BS Oscillations Assumption}: \label{cond:decay_coefficients} $\limsup_{k \to \infty }\limsup_{n \to \infty} \left| \int_{\Xi_n} e^{2\pi k i S_{h_n} \varphi (x) } d\mu(x) \right| < \delta$.
\end{enumerate}
Then the associated skew-product $T_\varphi $ is ergodic. 
\end{proposition}
\noindent The criterion stated above is proved in {[Proposition 2.3, \cite{fraczek_ergodic_2012}]} (see also {[Proposition 12, \cite{lemanczyk_ergodic_1996}}). Assumption \ref{tightness_assumption} guarantees that the sequence of distributions of Birkhoff sums $(S_{h_n} \varphi |_{\Xi_n})_{n\in\N}$, when seen as random variables where the variable $x$ is uniformly distributed on $\Xi_n$, is \emph{tight}, which guarantees \emph{tightness of the Birkhoff sums} along partial rigidity sets. Assumption \ref{oscillations_assumption} can be verified by showing that the Birkhoff sums $S_{h_n} \varphi $ display sufficient oscillations; from this, the name \emph{Birkhoff sums oscillations}.

In our case, the sets $(\Xi_n)_{n \in \N}$ in the criterion will be given by appropriate Rohlin towers, which we obtain as subsets of the towers provided by the Rauzy-Veech induction procedure. By picking the induction times appropriately, the measure of these towers will be bounded from below (so that Condition \eqref{cond:towers_measure} of Definition~\ref{def:partial_rigidity} holds), and they will be rigid (in the sense of Conditions \eqref{cond:quasi_invariance}, \eqref{cond:partial_rigidity} of Definition~\ref{def:partial_rigidity}). Furthermore, one can guarantee that iterates of points in the tower up to rigidity times do not come too close to singularities. 

To {verify the Assumption}~\ref{cond:tightness}, { it is crucial to exploit that the logarithmic singularities are \emph{symmetric}} (otherwise, when the singularities are asymmetric, there cannot be tightness, see, e.g., \cite{Ulcigrai_mixing_2007, Rav:mix, Sch:mix}). To prove tightness, we rely on the cancellations proved in \cite{ulcigrai_absence_2011}, 
which guarantee that Birkhoff sums of the derivatives can be controlled far from singularities. We also need to control some values of the {cocycle} (at { midpoints of each continuity interval}). The latter control is achieved by exploiting the {symmetry of the permutation} combined with the assumption { that the cocycle is \emph{odd}}.

Finally, to prove that together with tightness (i.e., Assumption \ref{tightness_assumption}) there are also oscillations (i.e., Assumption \ref{oscillations_assumption} holds), which are crucial to guarantee ergodicity (since cocycles which are coboundaries are trivially tight, but do not produce ergodicity), it will be helpful to have a quantitative control \emph{from below} of the derivative of Birkhoff sums on the sets $(\Xi_n)_{n \in \N}$. This can be achieved (see Theorem \ref{thm:derivativeestimates}) guaranteeing that the orbits up to the rigidity time of the points in the rigidity set $\Xi_n$ do get close (although in a controlled fashion) to discontinuities on which the cocycle $\varphi $ has a non-trivial logarithmic singularity. 

\section{Preliminaries and reduction}

First, let us define several notions we informally mentioned in the introduction (particularly the class of cocycles we consider). We then show in \S~\ref{sec:reduction} how to deduce the result about extensions of locally Hamiltonian flows in genus two from the result on skew-products.

\subsection{Basic definitions}\label{sec:defs}
\subsubsection{Symmetric interval exchange maps}\label{sec:IETs}

 Let $\mathcal{A}$ be a $d$-element alphabet. An interval exchange transformation $T = (\pi, \lambda)$ is a piecewise isometry of the interval $I\subset\R$ determined by a pair $\pi=(\pi_0,\pi_1)$ (to which we refer to as a \emph{permutation}) of bijections $\pi_\vep:\mathcal{A}\to\{1,\ldots,d\}$, for $\vep=0,1$, and a vector $\lambda=(\lambda_\alpha)_{\alpha\in\mathcal{A}}\in \R_{>0}^{\mathcal{A}}$. For any $\lambda=(\lambda_\alpha)_{\alpha\in\mathcal{A}}\in \R_{>0}^{\mathcal{A}}$, let
\[|\lambda|=\sum_{\alpha\in\mathcal{A}}\lambda_\alpha,\quad I=\left[0,|\lambda|\right),\]
 and define
\[I_{\alpha}=[l_\alpha,r_\alpha),\qquad \text{ where } \qquad l_\alpha=\sum_{\pi_0(\beta)<\pi_0(\alpha)}\lambda_\beta,\;\;\;r_\alpha
=\sum_{\pi_0(\beta)\leq\pi_0(\alpha)}\lambda_\beta.\] 
With these notations, $|I_\alpha|=\lambda_\alpha$. We consider only \emph{irreducible} permutations, that is, only pairs of bijections $\pi_1, \pi_2: \A \to \{1, \dots, d\}$ for which $\pi_1\circ\pi_0^{-1}\{1,\dots,k\}=\{1,\ldots,k\}$ implies $k=d$. We denote by $S^{\mathcal A}$ the set of all irreducible permutations. Denote by $\Omega_\pi$ the matrix
$[\Omega_{\alpha\,\beta}]_{\alpha,\beta\in\mathcal{A}}$ given by
\begin{equation}\label{eq:Omega}
\Omega_{\alpha\,\beta}=
\left\{\begin{array}{cl} +1 & \text{ if
}\pi_1(\alpha)>\pi_1(\beta)\text{ and
}\pi_0(\alpha)<\pi_0(\beta),\\
-1 & \text{ if }\pi_1(\alpha)<\pi_1(\beta)\text{ and
}\pi_0(\alpha)>\pi_0(\beta),\\
0& \text{ in all other cases.}
\end{array}\right.\end{equation}
Then, the IET $T_{(\pi,\lambda)}$ associated with $(\pi,\lambda)$ is given by $T_{(\pi,\lambda)}x=x+w_\alpha$, for any $x\in I_\alpha$, where $w=\Omega_\pi\lambda$. 

\smallskip
 Let $End(T)$ stand for the set of endpoints of the intervals $\{I_\alpha\}_{\alpha\in\mathcal{A}}$. A pair ${(\pi,\lambda)}$ satisfies the {\em Keane condition} if $T_{(\pi,\lambda)}^m (l_{\alpha})\neq l_{\beta}$, for all $m\geq 1$, and for all $\alpha,\beta\in\mathcal{A}$ with $\pi_0(\beta)\neq 1$. Keane \cite{keane_interval_1975} showed that an IET with an irreducible permutation that satisfies this condition is minimal.

We often consider the case of normalized IETs, that is when $|I|=1$. For this purpose, we denote by $\Lambda^{\mathcal A}$ the positive simplex of all vectors $\la\in\R^{\A}_{>0}$ such that $|\la|=1$.

\smallskip
\noindent{\it Symmetric permutations.} In this paper, we will be mostly concerned with \emph{symmetric combinatorics} (which fully reverse the order of the intervals), i.e., such that $$\pi_{1}\circ {\pi_0^{-1}}(i)={d} - i + 1,$$ for every $1\leq i\leq d$. 
Notice that 
when $\pi$ is the symmetric permutation, the discontinuities of $T$ and $T^{-1}$, which we denote by $0 = u_0^t < u_1^t < \dots < u_d^t = 1$ and $0 = u_0^b < u_1^b < \dots < u_d^b = 1$, respectively, are given by
\begin{equation}
\label{eq:discontinuities_symmetric}
u_k^t = \sum_{\pi_0(\alpha) < k} \lambda_\alpha, \quad\quad u_k^b = \sum_{\pi_1(\alpha) < k} \lambda_\alpha = \sum_{\pi_0(\alpha) \geq d + 1 - k} \lambda_\alpha {= 1 - u_{d - k}^t}.
\end{equation}
{Thus, for any ${1} \leq k \leq d$, we have that
\begin{equation}
\label{eq:image_discontinuities_symmetric}
T(l_{\pi_0^{-1}(k)}) = T(u_{k - 1}^t) = u_{d - k}^b.
\end{equation}}

\subsubsection{{Odd} cocycles with logarithmic singularities}\label{sec:logsing}

Let $\A$ be an alphabet of $d \geq 3$ elements and consider a partition $I = \bigsqcup_{\alpha \in \A} I_\alpha$ of $I$ into a finite number of subintervals. Throughout this work, we will be interested in the partition $I = \bigsqcup_{\alpha \in \A} I_\alpha$ given by the intervals exchanged by an IET. 
We denote by $C^r(\bigsqcup_{\alpha \in \A} I_\alpha)$ the space of functions $\varphi :I \to \R$ such that $\varphi \mid_{I_\alpha}$ is of class $C^r$, for each $\alpha \in \A$. Similarly, we denote by $AC(\bigsqcup_{\alpha \in \A} I_\alpha)$ (resp. $BV(\bigsqcup_{\alpha \in \A} I_\alpha)$) the space of functions whose restrictions to each interval are absolutely continuous (resp. of bounded variation).

\medskip
{\noindent \emph{Symmetric logarithmic singularities.}
We will consider the following class of singular cocycles:}
\begin{definition}[Logarithmic cocycles]\label{def:logsing}
 We say that the cocycle $\varphi \in C(\bigsqcup_{\alpha \in \A} I_\alpha)$ has \emph{logarithmic singularities} if there exist constants $C_\alpha^+, C_\alpha^- \in \R$, for each $\alpha \in \A$, and a function $g_\varphi \in AC(\bigsqcup_{\alpha \in \A} I_\alpha)$ such that $g_\varphi ' \in BV(\bigsqcup_{\alpha \in \A} I_\alpha)$ and 
\begin{equation}
\label{eq:logartihmic_singularities}
 \varphi (x) = g_\varphi (x) + \sum_{\alpha \in \A} C_\alpha^+ \log\left(|I|\left\{\frac{(x - l_\alpha)}{|I|}\right\}\right) + C_\alpha^- \log\left(|I|\left\{\frac{(r_\alpha - x)}{|I|}\right\}\right).
\end{equation}
\noindent Furthermore, we say that the logarithmic singularities of $\varphi $ are \emph{non-trivial} if 
\[C_{\pi_0^{-1}(1)}^+ + C_{\pi_0^{-1}(d)}^+ \neq 0 \quad \text{ or } \quad C_{\pi_0^{-1}(k)}^+ \neq 0, \text{ for some } 1 < k < d.\] 
\end{definition}
\noindent 
The non-triviality condition guarantees that at least one branch has a nontrivial logarithmic singularity. The reason for the formulation (in particular, { the first of the above conditions that asks for a \emph{sum} of two constants to be non-zero) is technical (and related to the special role played by the preimage of the endpoint $0$ as discontinuity)}. Still, it is easy to check in the class cocycles with a geometric origin (i.e., appear from the reduction of an extension of a locally Hamiltonian flow to a skew-product). 
 
\begin{definition}\label{def:longsingspace}
We denote by $LS^r(\bigsqcup_{\alpha \in \A} I_\alpha)$ the space of functions in $C^r(\bigsqcup_{\alpha \in \A} I_\alpha)$ having \emph{non-trivial} logarithmic singularities. 
\end{definition}

\begin{definition}[Symmetric logarithmic cocycles]\label{def:symlog}
Given $\varphi \in LS^r(\bigsqcup_{\alpha \in \A} I_\alpha)$ as in \eqref{eq:logartihmic_singularities}, we say that its logarithmic singularities are \emph{symmetric} if
\begin{equation}
\label{eq:symmetric_singularities}
 \sum_{\alpha\in \A} C_\alpha^- = \sum_{\alpha\in \A} C_\alpha^+.
\end{equation}
We denote by $LSS^r(\bigsqcup_{\alpha \in \A} I_\alpha)$ the space of functions with logarithmic symmetric singularities. Otherwise, we say that the logarithmic singularities are \emph{asymmetric}. 
\end{definition}
\noindent 
{Note that in the locally Hamiltonian flows literature (see e.g.~\cite{ulcigrai_absence_2011, fraczek_asymptotic_2021}), when logarithmic singularities are considered, it is usually for roofs of special flows, so that the constants $C_\alpha^\pm$ are all positive (or non-negative). We stress that in this definition for \emph{cocycles} (as in \cite{fraczek_ergodic_2012, fraczek_asymptotic_2021}), the constants $C_\alpha^\pm$ are real numbers, not necessarily positive.}

\medskip
{\noindent \emph{Odd cocycles.}}
Let us now define the class of cocycles that we call \emph{odd on each interval}. Given an interval $I = [a, b] \subseteq \R$, we denote by $\mathcal{I}_I:I \to I$ the involution given by
\begin{equation}\label{eq:intervalI}\mathcal{I}_I(x) = a + b - x, \qquad \text{for\, all}\ x\in I.
\end{equation}
{ One can check{ , by a direct calculation, that any symmetric IET $T: I \to I$ verifies} \begin{equation}
\label{eq:characterization_symmetric}
T^{-1} \circ \mathcal I_I = \mathcal I_I \circ T.
\end{equation}
}
\begin{definition}[Cocycles odd on each {continuity} interval]\label{def:odd}
A cocycle $\varphi $ in $LS^r(\bigsqcup_{\alpha \in \A} I_\alpha)$ is \emph{odd with respect to each {continuity} interval} if the cocycle $\varphi $ restricted to each $I_\alpha $ is an odd symmetric function with respect to the middle point of the restriction interval, i.e.,
\begin{equation}\label{eq:Iinv}
\varphi (x)=-\varphi ( \mathcal{I}_{I_\alpha}(x)) \qquad \text{for\ all}\ x\in I_\alpha, \, \alpha \in \mathcal{A}.
\end{equation}
\end{definition}
\begin{figure}[h]\label{fig:odd_cocycle}
	\centering
		\includegraphics{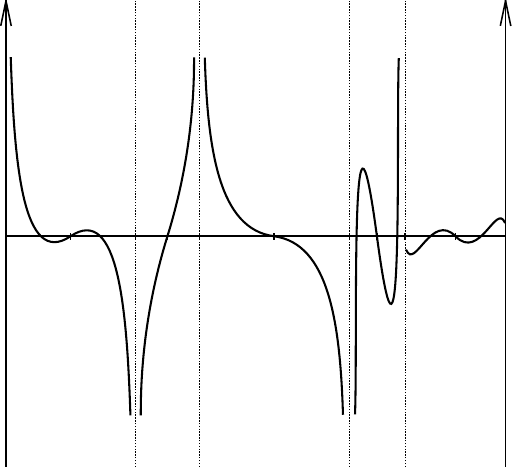}
		\captionsetup{width=0.75\linewidth}
		\caption{\small An example of a cocycle over five intervals with non-trivial logarithmic singularities, which is odd {on} each of them.}
\end{figure}
\smallskip

Given an IET $T: I \to I$ with a \emph{symmetric} permutation, { the oddness condition stated above can be equivalently rephrased as a request that the cocycle has the following behavior under the \emph{involution} $\mathcal{I}_{I}$ of the domain of definition $I$ of $T$.}

\begin{lemma}[Oddness and involution, see also \cite{chaika_singularity_2021}]\label{lemma:involutionodd}
\label{lem:odd_cocycle}
Let $T: I \to I$ be an IET with a symmetric permutation. A cocycle $\varphi : I \to \R$ is odd on each continuity interval of $T$ if and only if 
\begin{equation}
\label{eq:oddly_symmetric}
\varphi \circ T^{-1} = - \varphi \circ \mathcal{I}_I.
\end{equation}
\end{lemma}
\noindent {The proof is based on a simple calculation, which also appeared in \cite{chaika_singularity_2021}, and we include here to underline the role of symmetry.}
\begin{proof}
Let $\alpha \in \A$ and $x \in I_\alpha$. {Let us first show that since $T$ has a symmetric permutation
\begin{equation}\label{eq:symmetrypi} \mathcal{I}_{I_\alpha}(x) = \mathcal{I}_I \circ T(x).\end{equation}
Indeed, if $0\leq i<d$ is such that $I_\alpha=[u^t_i,u^t_{i+1})$, $T$ maps linearly $I_\alpha$ onto $[|I|-u^t_{i+1},|I|-u^t_{i})$, i.e.,
$T(x)=x+|I|-u^t_i-u^t_{i+1}$
for all $x\in [u^t_i,u^t_{i+1})$. 
Thus, one can verify directly that \eqref{eq:symmetrypi} holds, since if $x\in [u^t_i,u^t_{i+1})$, 
$$\mathcal{I}_I(T(x))=|I| - T(x)= |I|- x- |I|+u^t_i+u^t_{i+1}= u^t_i+u^t_{i+1} - x= \mathcal{I}_{I_\alpha}(x),$$
where the last equality follows by the definition \eqref{eq:intervalI} of $\mathcal{I}_{I_\alpha}$.
\smallskip
We now claim that a measurable function $\phi: I\to\R\cup\{\pm\infty\}$ is odd on each interval if and only if it satisfies \eqref{eq:oddly_symmetric}, which precomposing with $T^{-1}$ can be equivalently rewritten as $\varphi =-\varphi \circ\mathcal{I}\circ T^{-1}$. 
Indeed, this expression 
evaluated at $x \in I_\alpha$, in view of \eqref{eq:symmetrypi}, is equivalent to
\[ \varphi (x) = - \varphi \circ \mathcal{I}_I \circ T(x) = -\varphi (\mathcal{I}_{I_\alpha}(x)),\]
which is the definition of an odd cocycle with respect to each interval. }
\end{proof}

\subsubsection{Rauzy-Veech induction notation}\label{sec:RV}

The classical induction procedure for IETs, known as the \emph{Rauzy-Veech induction}, is fundamental to the study of IETs. 
We assume the reader is familiar with it and refer otherwise to \cite{viana_ergodic_2006} or \cite{yoccoz_interval_2010} for a detailed introduction to the tools. Here we briefly recall the notation that we will use. 

\smallskip
\noindent {\it Notation for iterates.} Given an IET $T$ of $d\geq 2$ intervals corresponding to $(\pi, \lambda)$ and verifying Keane's condition, one can define, for any $n\in \N$, the iterates of $T$ under the Rauzy-Veech renormalization $\RV$, which are obtained by inducing $T$ on a suitably chosen sequence of nested subintervals shrinking towards the left endpoint of {$I^{0}$}$:=I$, selected so that the induced map is an IET of the same number $d$ of exchanged intervals and then rescaling the resulting map so that it is again defined on $I$. 
 
For any $n \geq 0$, we denote by ${I^n:=}I^n(T)$ the $n$-{th} inducing subinterval, by $T_n = \big(\pi^n, \lambda^n\big)$ the IET obtained inducing $T$ on $I^n(T)$, i.e., the first return map of $T$ to $I^n\subseteq I$. We denote by $\{{I^n_\alpha:=}I^n_\alpha(T)\}_{\alpha\in\A}$ the intervals exchanged by $T_n$ and by $ |I^n|$ the length of the inducing interval. The $n$-{th} iterate of Rauzy-Veech renormalization of $T$ is then given by
\[
\RV^n(T)(x):= \frac{T_n ({|I^n|} x)}{{|I^n|}}, \qquad \forall x \in I.
\] 
Finally we denote by $q^n = q^n(T) = (q^n_\alpha)_{\alpha \in \A}$ the vector in $\mathbb{N}^{\A}$ whose entries $q^n_\alpha$, $\alpha\in \A$, give the \emph{return time} of $I^n_\alpha$ { to } $I^n$ under the action of $T$. If there is no risk of confusion, we will omit the explicit dependence on $T$ in all of the above notations. 

Any finite sequence of permutations which {can be obtained as successive permutations of an orbit of the} Rauzy-Veech algorithm is called a \emph{Rauzy path} $\gamma$ of { combinatorial length $|\gamma|$ (where the combinatorial length is the number of permutations minus $1$)}. With every Rauzy path $\gamma$ we can associate a non-negative $d\times d$-matrix $A_\gamma$, called \emph{Rauzy matrix}, which satisfies the following: if $(\pi,\lambda)$ is such that $(\pi^i)_{i = 0}^{|\gamma|-1} = \gamma$, then $\la=A_{\gamma}\la^n$. An important use of this matrix comes from the fact that if $A_\gamma$ is a positive matrix for some Rauzy path $\gamma$, then there exists a constant $\rho(\gamma)>0$ such that, for any IET $(\pi,\la)$ verifying Keane's condition and any $n\ge|\gamma|$ such that $(\pi^i)_{i=n-|\gamma|+1}^{n} = \gamma$, we have
\begin{equation}\label{eq: matrixbound}
\max_{\alpha, \beta \in \A}\frac{q_\alpha^n}{q_\beta^n}<\rho(\gamma).
\end{equation}

\smallskip
\noindent {\it Polygons and zippered rectangles.} 
Given an IET $(\pi,\lambda)$ one can add a \emph{suspension datum} $\tau\in \Theta^\pi$, where
\begin{equation}
\label{eq:defTheta}
\Theta^\pi:=\left\{\tau\in\R^{\A}\,\left|\,\quad 
\sum_{\pi_1(\alpha)<k}\tau_{\alpha}<0< \sum_{\pi_0(\alpha)<k}\tau_{\alpha},\ 
\text{ for every }1< k\le d\right\}\right. .
\end{equation}
By drawing two broken line segments starting at 0 in $\mathbb C$, one created by taking the consecutive line segments $\la_{\pi_0^{-1}(i)}+i\tau_{\pi_0^{-1}(i)}$ and the other by taking $\la_{\pi_1^{-1}(i)}+i\tau_{\pi_1^{-1}(i)}$, and identifying the subsegments which correspond to the same symbol $\alpha\in\mathcal A$, we get a \emph{polygonal representation of a translation surface}. We denote by 
		\[
		\mathcal M:=\{(\pi,\lambda,\tau)\ |\ \ \pi\in S^{\A},\ \lambda\in \R^{\mathcal A}_{>0},\ \tau\in\Theta^\pi \}, \]
the space of all polygonal representations. { Notice that some of these polygonal representations may have self-intersections. However, in this paper, we will only use polygonal representations associated with triples for which this does not happen.} We also consider the space of \emph{normalized} polygonal representations
\begin{align*}
\tilde{\mathcal M} :=\{(\pi,\lambda,\tau)\in 	{\mathcal M}\ |& \ \ |\la|=1 \text{ and } Area((\pi,\lambda,\tau))=1 \}, \\ & { \text{where} \ Area((\pi,\lambda,\tau)):=\sum_{\alpha\in\mathcal{A}}\lambda_\alpha \left( \sum_{\pi_0(\beta)<\pi_0(\alpha)}\tau_{\beta}\right).
}		
\end{align*}
{To a triple $(\pi,\lambda,\tau)$ one can also associate a \emph{zippered 
rectangle} (see \cite{yoccoz_interval_2010}). These are {flat rectangles 
embedded in the translation surface} which have the intervals $I_{\alpha}$ as 
bases and whose heights are given by the vector 
$h=-\Omega_{\pi}\tau\in\R^{\mathcal A}_{>0}$.} The space of zippered rectangles 
constitutes another natural domain for the natural extension of Rauzy-Veech 
induction and provides another representation of the underlying translation 
surface. For a precise definition of zippered rectangles, see 
\cite{viana_lectures_2014} or \cite{yoccoz_interval_2010}. 

{
On the space $\tilde{\mathcal M}$ of triples $(\pi,\lambda,\tau)$ one can define the \emph{extended Rauzy-Veech induction} $\mathcal{R}: \tilde{\mathcal M}\to \tilde{\mathcal M}$, which provides an invertible extension of $\mathcal{R}: \mathcal{M}\to \mathcal{M}$. At the level of polygonal representations, the action of 
 \emph{extended Rauzy-Veech induction} $\mathcal R$ consists of cutting a 
 properly chosen triangle of the polygon $(\pi,\lambda,\tau)$ and gluing it to 
 one of the intervals corresponding to the second copy of one of its sides. 
 In this case, both polygons, before and after induction, represent the same translation surface. We denote by $\mathcal R^n(\pi,\lambda,\tau)=(\pi^n,\lambda^n,\tau^n)$, for any $n \in \N$, the triples in the orbit of $(\pi,\lambda,\tau)$ under $\tilde{\mathcal M}$. In particular, since this is an extension of Rauzy-Veech induction, $(\pi^n,\lambda^n)=T_n$.}

 One can also consider the \emph{normalized} extended Rauzy-Veech induction on $\tilde{\mathcal M}$ { as follows. Given $\lambda\in\R_+^\mathcal{A}$, let us denote by $|\lambda|:=\sum_{\alpha}\lambda_\alpha $. Then} 
 	\[
 	\tilde{\mathcal R}(\pi,\lambda,\tau):=\left(\pi^1,\frac{\lambda^1}{|\lambda^1|},|\lambda^1|\cdot\tau^1\right).
 	\]
{The normalized extended Rauzy-Veech induction is invertible and extends the 
normalized Rauzy-Veech induction $\tilde{\mathcal{R}}:\mathcal{M}\to 
\mathcal{M}$.} The importance of this extension comes from a classical theorem 
by Masur \cite{masur_interval_1982} and Veech \cite{veech_gauss_1982}, which 
states that this transformation is ergodic with respect to a probability 
measure equivalent to a product of the counting measure on $S^{\mathcal A}$ and 
{the} Lebesgue measure on $\Lambda^{\mathcal A}$ and $\Theta^{\pi}$.

\subsection{From extensions to skew-products}\label{sec:reduction}{
In this section, we recall how to reduce the study of extensions of locally Hamiltonian flows to the study of skew-products over IETs given by cocycles with logarithmic singularities and deduce Theorem~\ref{thm:ergodic_extension} from Theorem~\ref{thm:reductiontoIETs}. First, let us recall some general facts about the reduction of extensions of smooth area-preserving flows to skew products. For a proof of these facts, see, for example, \cite{chaika_singularity_2021} or \cite{fraczek_ergodic_2012}.

\smallskip
\noindent \emph{Skew product representations.} Given a locally Hamiltonian flow $(\phi_t)_{t \in \R}$ { of class $C^2$} on a compact orientable surface $M$, let $\Sigma$ denote the set of fixed points of $(\phi_t)_{t \in \R}$. We assume that the cardinality $\kappa$ of $\Sigma$ is finite and that all points in $\Sigma$ are saddle points. As recalled in the introduction, one can find a translation structure on $M$ such that $(\phi_t)_{t \in \R}$ is a (singular) time-change of the vertical linear flow of this translation surface; the conical singularities of the flat metric are exactly at points of $\Sigma$. It is well known that if $\mathscr{S} \subseteq M \setminus \Sigma$ is a transverse section {with endpoints on separatrices}, choosing suitable coordinates (so that the horizontal flow has unit speed), the Poincaré section is an IET $T$ of $d=2g+\kappa-1$ intervals. Moreover, given a continuous observable $f:M \to \R$, $\mathscr{S} \times \R$ is a section for the extension $(\phi^f_t)_{t \in \R}$; the Poincaré first return map of $(\phi^f_t)_{t \in \R}$ to $\mathscr{S} \times \R$ is a skew-product of the form 
\begin{equation}
\label{eq:skew-product}
 \Function{T_\varphi }{I \times \R}{I \times \R}{(x, t)}{(T(x), t + \varphi (x))},
\end{equation}
where $T$ is the IET obtained above as as Poincaré section of $(\phi_t)_{t \in \R}$ to $\mathscr{S} $ and $\varphi :I \to \R$ is a { $C^2$} function associating to each $x \in I$ the integral of the observable $f$ along the orbit segment defined by $x$ and $T(x)$. More precisely, if $r(x)$ is the first return time of $x$ to $\mathscr{S}$ under $(\phi_t)_{t\in\R}$, so that $T(x) = \phi_{r(x)}(x)$, then 
\begin{equation}\label{eq:returntime} 
\varphi (x) = \int_0^{{r(x)}} f(\phi_s(x))\,ds.
\end{equation}
One can show that if the saddles in $\Sigma$ are all non-degenerate, then $\varphi $ is a cocycle with symmetric logarithmic singularities. Furthermore, if $f(p)\neq 0$ for at least one saddle $p\in \Sigma$, the singularities are non-trivial in the sense of Definition~\ref{def:logsing}, so that $\varphi$ belongs to the space $LSS^2(\bigsqcup_{\alpha \in \A} I_\alpha)$ in Definition~\ref{def:symlog}, where $(I_\alpha)_{\alpha \in \A} $ are the intervals exchanged by $T$.

\begin{figure}[h]
	\centering
	\includegraphics[scale=0.8]{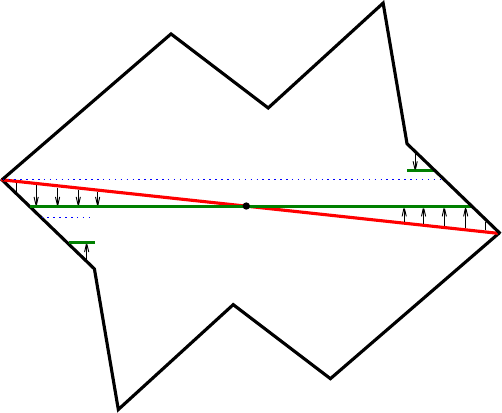}
	\captionsetup{width=0.6\linewidth}
	\caption{\small \label{fig:tilting} The diagonal can be tilted to obtain a 
	{horizontal} section without changing the first return map.}
\end{figure}

\smallskip
We are now ready to prove Theorem \ref{thm:ergodic_extension} from Theorem~\ref{thm:reductiontoIETs}. 
\begin{proof}[Proof of Theorem~\ref{thm:ergodic_extension}] 
Consider, as in the assumptions of Theorem~\ref{thm:ergodic_extension}, 
a surface 
$M$ of genus two and assume that $(\phi_t)_{t \in \R}$ is a locally Hamiltonian flow with two non-degenerate saddles. 
Let $\alpha:M\to\mathbb{R}$ be such that $(\phi_t)_{t \in \R}$ is generated by 
the vector field $\alpha X_v$, where $X_v$ is the { constant unit 
vertical vector field} on $M$ with { respect to the} given 
translation structure (see 
\S~\ref{sc:loc_ham}). 
Let $\mathcal{I}:M\to M$ be the underlying hyperelliptic involution (see \S~\ref{sc:loc_ham}) and assume that $f: M\to \R$ is, as in the assumptions of Theorem~\ref{thm:ergodic_extension}, such 
that $ \alpha f \circ \mathcal{I} = - \alpha f$. 
{ Let us now show that we can find a Poincaré section ${\mathscr{S}}$ which is fixed by the involution, i.e., such that $\mathcal{I}(\mathscr{S})=\mathscr{S}$, so that the Poincaré map to this section is a symmetric IET.} { By the Kontsevich-Zorich classification of connected components 
of the moduli space of translation structures (see Theorem 2 in \cite{kontsevich_connected_2003}), the translation surface $M$ can be represented as $(\pi,\lambda,\tau)\in \{\pi\}\times \R^{\mathcal A}_{>0}\times \Theta^\pi$, where $\pi$ is a symmetric permutation on an alphabet $\mathcal A$ of 5 elements {(here $\tau\in \Theta^\pi$ is a suspension datum, see \eqref{eq:defTheta} for the definition of $\Theta^\pi$). In the polygonal representation given by $(\pi,\lambda,\tau)$ (see Section~\ref{sec:RV}), consider 	as a Poincaré} section the diagonal ${\mathscr{S}}_0$ 
 connecting the points $(0,0)$ and $\big(\sum_{\alpha\in\A}\la_{\alpha}, \sum_{\alpha\in\A}\tau_{\alpha}\big)$ {(red in figure~\ref{fig:tilting})}. One can show that the diagonal is contained in the polygon (if not, one could use the involution $\mathcal{I}$ to contradict the definition of suspension datum $\mathcal{\tau}$, see \eqref{eq:defTheta}). Note also that the mid-point of the diagonal $\mathscr{S}_0$, 
namely the point $\big(\frac{1}{2}\sum_{\alpha\in\A}\la_{\alpha}, 	\frac{1}{2}\sum_{\alpha\in\A}\tau_{\alpha}\big)$, is the center of the polygon and one of the points fixed by the involution $\mathcal I$.}	

The Poincaré first return map to ${\mathscr{S}}_0$, which is given by the identification of the polygon sides (see Figure~\ref{fig:tilting}) is by construction\footnote{
The diagonal $\mathscr{S}_0$ can be tilted by sliding it along the leaves of the vertical foliation, using an argument analogous to those explained { above while leaving $(0, 0)$ fixed}, 
to see that the Poincaré map to the diagonal $\mathscr{S}_1$ is the same than the Poincaré map to the section given by the horizontal segment $[0,\sum_{\alpha\in\A}\la_{\alpha}) {\times \{0\}} $ (the horizontal segment dashed in figure~\ref{fig:tilting}), which, by definition of zippered rectangles, is the IET { given by $(\pi, \lambda)$}.} the symmetric IET with length $\lambda$ and combinatorial datum $\pi$. We now claim that we can \emph{tilt} ${\mathscr{S}}_0$ (keeping the center fixed) to obtain the desired section $\mathscr{S}$, which will be a horizontal segment passing through the center and symmetric with respect to it (hence in particular fixed by the involution $\mathcal{I}$). Notice that tilting the section {by \emph{sliding} it} along the leaves of 
	the foliation associated with the vertical translation { flow (so that the endpoints move along leaves)} does not change the 
	first return transformation via $(v_t)_{t\in\R}$ { as long as no singularity is hit. We claim that this is the case if we rotate the diagonal around the center, sliding its endpoints along leaves of the vertical flow, until it becomes a horizontal segment (the segment $\mathscr{S}$ drawn in green in Figure~\ref{fig:tilting}). To see this, assume without loss of generality that $\sum_{\alpha\in\A}\tau_\alpha<0$ (as in figure~\ref{fig:tilting}), the other case being similar. If a singularity were hit, applying the involution, we would find that $\tau$ does not satisfy the inequalities in the definition of $ \Theta^\pi$, see \eqref{eq:defTheta}. } 
	Then $\mathcal{I}(\mathscr{S})= \mathscr{S}$ and the Poincaré first return map on $\mathscr{S}$ is still the same symmetric IET $T$ of 5 intervals with combinatorics $\pi$. In particular, in view of 	\eqref{eq:characterization_symmetric}, it satisfies $T(x) = T^{-1}(\mathcal{I}(x))$.
 
 Let us now show that the induced cocycle $\varphi $ given by \eqref{eq:returntime} is odd with respect to the involution, i.e., it satisfies \eqref{eq:oddly_symmetric} and hence, equivalently (by Lemma~\ref{lemma:involutionodd}), is odd on each interval $I_\alpha,$ $\alpha\in \mathcal{A}$, in the sense of Definition~\ref{def:odd}.

 \begin{figure}[h]
 	\centering
 	\includegraphics[scale=0.5]{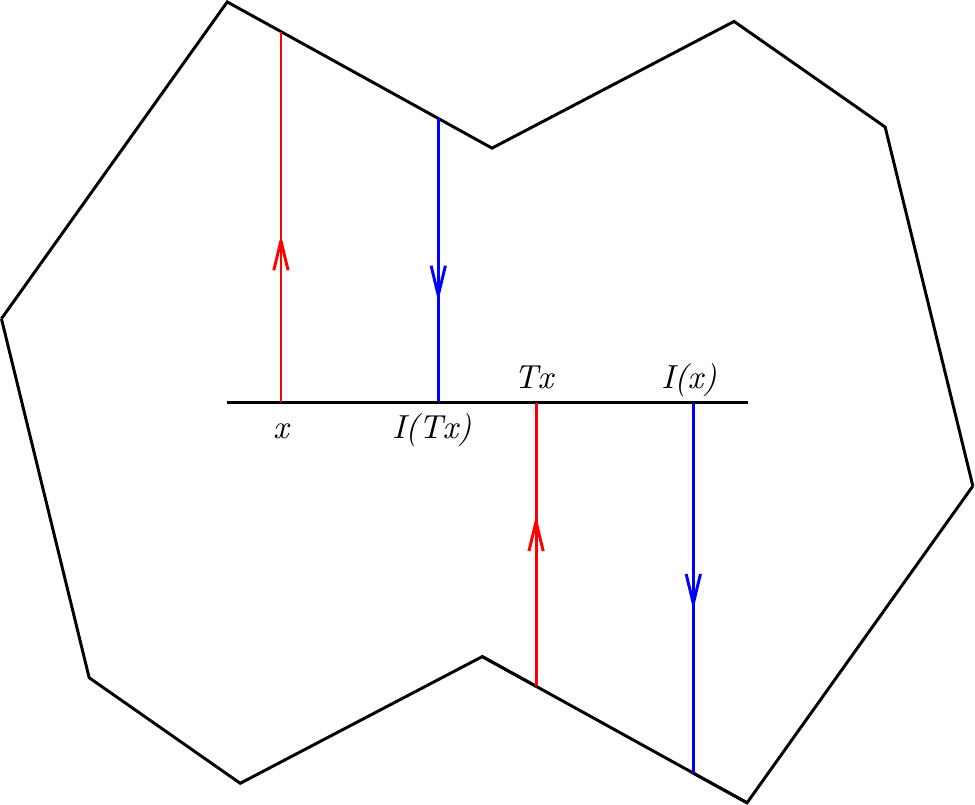}
 	\captionsetup{width=0.6\linewidth}
 	\caption{\small \label{fig:involution} The backward orbit of $\mathcal 
 	I(x)$ is 
 		the image of the forward orbit of $x$ via $\mathcal I$ for a symmetric 
 		Poincar\'e section.}
 \end{figure}
 
Recall that $(\phi_t)_{t \in \R}$ is a time change of the vertical flow $(v_t)_{t\in \R}$ generated by $X_v$, {which} means that
$$
\phi_{ \tau(x,t)} (x) = v_{ t}(x) \qquad \forall x \in M,\ t\in \R,
$$
where $\tau: M\times \R \to M$ is a cocycle over $(v_t)_{t\in \R}$, i.e., it satisfies,
\begin{equation}\label{eq:cocycle}
\tau(x, t_0+t)= \tau(x, t_0)+ \tau (v_{t_0}(x),t), \qquad \forall x \in M, \quad \forall t, t_0 \in \R, \end{equation}
and the infinitesimal generator $\alpha: X \to \mathbb{R}_{\geq 0}$ is given by
\begin{equation}
\label{eq:def_alpha}
\alpha (x):=\frac{\partial \tau}{\partial t}\bigg| _{t=0} \tau (x,t), \qquad \forall\ x\in M.
\end{equation}
From \eqref{eq:cocycle} and \eqref{eq:def_alpha}, it follows that for all $ t_0 \in \R $ and $x\in M$, 
\begin{align}
\frac{\partial \tau}{\partial t}\bigg| _{t=t_0} \tau (x,t)& = \frac{\partial \tau}{\partial t}\bigg| _{t=t_0} \left( \tau(x, t_0)+ \tau (v_{t_0}(x),t-t_0)\right) \nonumber \\ & = \frac{\partial \tau}{\partial s} \bigg| _{t=0} \tau (v_{t_0}(x),t) =\alpha (\phi_{\tau (x,{t_0})}). \label{eq:derivative_tau}
\end{align}
Given $x \in \mathscr{S}$ such that its first return $r(x)$ to $\mathscr{S}$ 
by the flow $(\phi_t)_{t \in \R}$ is well-defined, let $\rho(x)$ be the 
unique real number such that $r(x) = \tau(x, \rho(x))$. { Recall} 
that, in this 
case, the first returns to $\mathscr{S}$ by the flows $(\phi_t)_{t \in \R}$ 
and 
$(v_t)_{t \in \R}$ coincide, and we have
\[ \phi_{r(x)}(x) = T(x) = v_{\rho(x)}(x). \]
 By \eqref{eq:vertical_involution} applied to the point $T(x)$, 
{as represented in Figure~\ref{fig:involution},}
\begin{equation*}
\label{eq:involution_relation}
v_t( \mathcal I \circ{ T}(x)) = \mathcal I (v_{-t} \circ T(x)),
\end{equation*}
for any $t \in \R$ for which the flows above are well-defined. 
{Geometrically, as shown in Figure~\ref{fig:involution}, this shows 
that the image by the involution of the backward trajectory from $T(x)$ 
coincides with the forward trajectory from $\mathcal{I}(x)$. }
In particular, since $T(x) = v_{\rho(x)}(x)$ is the first return of $x$ to 
$\mathscr{S}$, it follows 
 that the first return of $\mathcal I \circ T(x)$ to $\mathscr{S}$ is given by 
$\mathcal I (x)$. More precisely, 
 \begin{equation}
 \label{eq:return_involution}
 \rho(\mathcal I \circ T(x)) = \rho(x), \qquad T( \mathcal I \circ T(x)) = 
v_{\rho(x)} ( \mathcal I \circ T(x)) = \mathcal I (x).
 \end{equation}
For any $x \in \mathscr{S}$ with finite return time $r(x)$ to $\mathscr{S}$,
\[
\varphi (x) = \int_0^{r(x)} f(\phi_s(x))\,ds = \int_0^{\tau(x, \rho(x))} f(\phi_s(x))\,ds.
\]
Applying the substitution $s = \tau(x, t)$ to the RHS of the previous equation and by \eqref{eq:derivative_tau}, 
\begin{equation}
\label{eq:roof_vertical_flow}
\varphi (x) = \int_0^{\rho(x)} \alpha f(\phi_{\tau(x, t)}(x))\,dt = \int_0^{\rho(x)} \alpha f(v_t(x))\,dt .
\end{equation}
By \eqref{eq:return_involution}, and by \eqref{eq:roof_vertical_flow} applied to $\mathcal I \circ T(x)$, 
\begin{align*}
\varphi (\mathcal I \circ T(x)) & = \int_0^{\rho(x)} \alpha f(v_t \circ \mathcal I \circ T(x))\,dt = \int_0^{\rho(x)} \alpha f(v_{\rho(x) - t} \circ \mathcal I \circ T(x))\,dt \\ & = \int_0^{\rho(x)} \alpha f(v_{- t} \circ v_{- \rho(x)} \circ \mathcal I \circ T(x))\,dt = \int_0^{\rho(x)} \alpha f(v_{-t} \circ \mathcal I (x))\,dt \\ & = \int_0^{\rho(x)} \alpha f(\mathcal I \circ v_t (x))\,dt = - \int_0^{\rho(x)} \alpha f(v_t (x))\,dt \\ & = - \varphi (x).
\end{align*}
Therefore, $\varphi $ verifies \eqref{eq:oddly_symmetric}.

Thus, we proved that the skew-product $T_\varphi $ that we obtain in this setting satisfies the assumptions of Theorem~\ref{thm:reductiontoIETs} and hence it is ergodic if $T$ belongs to a full measure set of IETs. Since the extension $(\phi^f_t)_{t \in \R}$ is ergodic if (and only if) its Poincaré map $T_\varphi $ is ergodic, Theorem~\ref{thm:ergodic_extension} now follows from the fact that the set of locally Hamiltonian flows whose Poincaré map on a given section $\mathscr{S}$ {(when seen as translation surfaces as described above)} belongs to a full measure set { of IETs} has full measure with respect to the measure class given by the Katok fundamental class. 
\end{proof}

\begin{remark}\label{rk:odd_observables}
Let $\mathcal{F}_{odd}\subseteq \mathcal{C}^2(M)$ be the class of observables $f:M\to \mathbb{R}$ such that there exists a section $\mathscr{S}\subseteq M \setminus \Sigma$ fixed by the underlying involution $\mathcal{I}$ of $M$, i.e., such that $\mathcal{I}(\mathscr{S})=\mathscr{S}$, with the property that, if ${r(x)}$ denotes the first return time of $x$ to $\mathscr{S}$, 
$$ \int_0^{r(x)} f(\phi_s(x))\,ds = 
- \int_0^{r(\mathcal I \circ T (x))} f(\phi_s(\mathcal{I}(T(x)))\,ds, $$
 that is, the cocycle $\varphi $ given by \eqref{eq:returntime} is odd with respect to the involution. 
Thus, the reduction shows, more generally, that if $f\in \mathcal{F}_{odd}$ then Theorem~\ref{thm:ergodic_extension} implies that the extension $(\phi^f_t)_{t \in \R}$ is ergodic. While we proved that any $f$ which satisfies the assumption \eqref{eq:odd_observable} of Theorem~\ref{thm:ergodic_extension} belongs to $\mathcal{F}_{odd}$, the class $\mathcal{F}_{odd}$ contain many more functions than those which satisfy \eqref{eq:odd_observable}. 
\end{remark}}

\section{Construction of the good rigidity sets}

In this section, we construct the rigidity sets $\Xi_n$ that we will use to verify the assumptions of the ergodicity criterion (Proposition~\ref{prop:ergodicity_criterion}) and hence prove the main result, Theorem~\ref{thm:reductiontoIETs}.

\subsection{Construction of good base dynamics}\label{sec:goodbase}

Let us start by showing how to construct, for any IET $T$ with a sufficiently small \emph{$\delta$-drift} (see Definition \ref{def:drift}) and sufficiently balanced lengths, subintervals with nice rigidity properties, whose distance to the discontinuities of $T$ is explicitly controlled. 

In the following, if $\pi$ is a symmetric permutation on $\A$, { let $\alpha\mapsto \overline{\alpha}$ be the involution of the alphabet, which inverts the order of the letters from left to right in the interval, namely for any $\alpha\in \mathcal{\alpha}$, since $\pi$ is a symmetric combinatorial datum, we define $\overline{\alpha}\in\mathcal{A}$ 
to be} 
$$\overline{\alpha} := \pi_{0}^{-1}(\pi_1(\alpha)) = \pi_{0}^{-1}(d + 1 - \pi_0(\alpha)) ,\qquad \text{for each} \, \alpha \in \A.$$ 
Remark that if all subintervals had equal length, i.e., $\lambda_\alpha= 1/d$ for all $\alpha\in \A$ (which is not compatible with the Keane condition) then, because of the symmetry assumption on the permutation, $T^2$ would be the identity and the intervals $I_\alpha$ would be mapped to $I_{\overline{\alpha}}$ by $T$, and vice versa. 

The IETs which we consider to build good dynamics (and that later we will use as base dynamics for Rohlin towers given by renormalization) {are} just a perturbation of this order two periodic behavior, chosen so that the iterates of a subinterval of $F_\alpha\subseteq I_\alpha$ when moving to $I_{\overline{\alpha}}$ and then back in $I_\alpha$, \emph{drift} to the right by a proportion $\delta >0$ of $|I|$. { To define the class of IETs more precisely, let us give two definitions.

\begin{definition}[Balanced lengths]\label{def:balance}
Given an IET $T = (\pi, \lambda)$ and $\nu > 0$, we say that $T$ is \emph{$\nu$-balanced} if $\lambda$ satisfies
\begin{equation}
\label{eq:balanced_vector}
\max_{\alpha, \beta \in \A} \frac{\lambda_\alpha}{\lambda_\beta} < \nu.
\end{equation} 
\end{definition}
\noindent In particular, $\nu$-balance implies that 
\begin{equation}\label{eq:balancedintervals} \frac{{ |I|}}{\nu d} \leq \lambda_\beta \leq \frac{\nu { |I|}}{d}, \quad \text{for all } \beta \in \A.
\end{equation}
The second definition characterizes the parameters which give the drifting phenomenon. }
\begin{definition}[$\delta$-drift]\label{def:drift}
Given an IET $T = (\pi, \lambda)$ with a symmetric permutation, $\delta > 0$ and $\alpha = \pi_0^{-1}(k)$, for some $1 < k \leq d,$ we say that $T$ has a \emph{$\delta$-drift at $\alpha$} if
\begin{equation}
\label{eq:delta_drift}
\delta |I| + u_{k - 1}^t \leq u_{k - 1}^b, \quad\quad \delta \min\{1, d - k \}|I| + u_{k}^t \leq u_{k}^b.
\end{equation} 
\end{definition}
\noindent {
\begin{remark}\label{rk:notempty}
One can see, by considering small perturbations of the (non-Keane) IET with all intervals of equal length $\tfrac{1}{d}$, that the set of IETs with $\delta$-drift is not empty if $\delta$ is sufficiently small, for example, if {$\delta< \tfrac{1}{2d}$.}
\end{remark}
When $T$ has sufficiently balanced lengths and a sufficiently small} $\delta$-drift at $\alpha$, we can then produce a subinterval $F_\alpha\subseteq I_\alpha$, with the good dynamics described in the following lemma (which will then be helpful for us to control rigidity and visits close to singularities): {the successive iterates $T^{i}(F_{\alpha})$, for $0\leq i\leq 5$ belong alternatively to $I_\alpha$ or $I_{\overline{\alpha}}$ (Condition \ref{cond:alternating_intervals} of Lemma~\ref{lem:rigid_interval}){. Also, $F_{\alpha}$ starts $\eta$-close to a singularity but, for $1 \leq i \leq 5$ if $\alpha \neq \pi_0^{-1}(d)$ or $2 \leq i \leq 5$ if $\alpha = \pi_0^{-1}(d)$, the iterates $T^i(F_\alpha)$} are all at least $\delta$-far from {any singularity} (as given by Conditions \ref{cond:not_last_interval} { or \ref{cond:last_interval}} of Lemma~\ref{lem:rigid_interval}).} {This is illustrated below in Figure~\ref{fig:goodbase}.}

\smallskip

	\begin{figure}[h]
	\centering
 	\captionsetup{width=0.75\linewidth}
		\includegraphics[scale=.9]{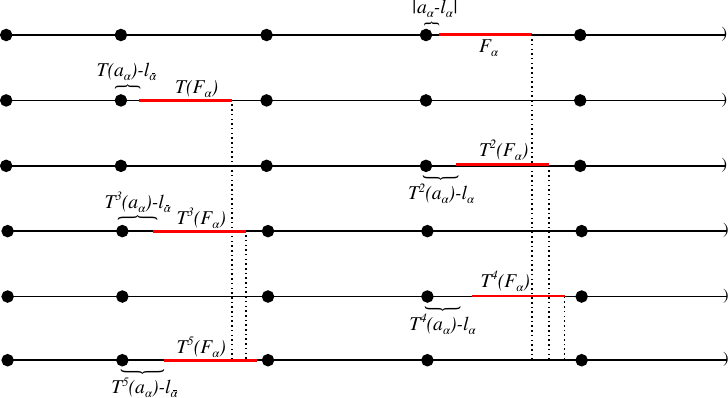}
	\caption{{\small { Pictorial representation of} good dynamics on the base in the case $\alpha \neq \pi_0^{-1}(d)$: { Rows show (in red) the location of the first six successive iterates of the subinterval $F_\alpha { = [a_\alpha, b_\alpha]}$ described in Lemma~\ref{lem:rigid_interval} (namely of $F_\alpha$ in the first row, $T(F_\alpha)$ in the second, $T^2(F_\alpha)$ in the third and so on) with respect to the discontinuities of $T$}. The interval $F_{\alpha}$ starts $\eta$-close to a singularity, {{ i.e., $|a_\alpha-l_\alpha|=\eta|I_\alpha|$}}, but its consecutive iterations are at least $\delta$-far from { any singularity} { (in particular 	$|T^i(a_\alpha)-l_\alpha|>\delta|I_\alpha|$)}. \label{fig:goodbase}}}
\end{figure}

\smallskip

\begin{lemma}[Good base dynamics from $\delta$-drift]
\label{lem:rigid_interval}
Given $1< \nu< 1 + \tfrac{1}{100d}$ and $0<\delta< \tfrac{1}{10d}$, if $T = (\pi, \lambda)$ is a symmetric IET with $\nu$-balanced lengths, and having a $\delta$-drift at $\alpha = \pi_0^{-1}(k)$, for some $1 < k \leq d$ then, for any $0 < \eta < \tfrac{1}{4}$, the interval 
\[F= F_\alpha = [a_\alpha, b
_\alpha] := [l_\alpha + \eta \lambda_\alpha, l_\alpha + \tfrac{3}{4}\lambda_\alpha] \subseteq I_\alpha,\] verifies the following:
\begin{enumerate}
 \item \label{cond:alternating_intervals} $F, T^2(F), T^4(F) \subseteq I_\alpha$ and $T(F), T^3(F), T^5(F) \subseteq I_{\overline{\alpha}}.$
 \item \label{cond:not_last_interval} If $k \neq d,$ then $|a_\alpha - l_\alpha| = \eta|I_\alpha|$ and 

 $$|T^i(a_\alpha) - l_\beta | \geq \delta,\qquad {\text{for\ any\ } 0\leq i\leq \beta \ \text{and}\ \beta\in \mathcal{A}, \ \text{such\ that}\ (i,\beta)\neq (0,\alpha);}$$
 \item \label{cond:last_interval} If $k = d$, then $|a_\alpha - l_\alpha| = \eta|I_\alpha| = |T(a_\alpha) - l_{\overline{\alpha}}| $ and
 $$|T^i(a_\alpha) - l_\beta | \geq \delta, \qquad {
 \text{for\ any\ } 0\leq i\leq \beta \ \text{and}\ \beta\in \mathcal{A}, \ \text{such\ that}\
(i,\beta)\neq (0, \alpha)\ \text{or}\ (1, \overline{\alpha}).}$$
\end{enumerate}
\end{lemma} 

\begin{remark} { The reason for the different treatment of the case $k=d$ is that when $\alpha= \pi_0^{-1}(d)$, i.e., when $I_\alpha$ is the last exchanged interval, $T(I_\alpha)$ is the first interval, so $T(l_\alpha)= l_{\overline{\alpha}}=u_{0}^t=0$ and $l_\alpha = u^t_{d - 1} = T^{-1}(0)$. Thus, in this case, the first image of the interval cannot have any \emph{drift} and has to be excluded.} 
\end{remark}

\begin{proof}
For the sake of simplicity { of notation, let us { assume that $|I|=1$ and} denote by} 
$$\delta(k) := \delta \min\{1, d - k \}.$$ 
By the symmetry of the permutation (see \eqref{eq:discontinuities_symmetric} and \eqref{eq:image_discontinuities_symmetric}), we have
\begin{align*}
 T(l_\alpha) - l_{\overline\alpha} &= u_{d - k}^b - u_{d - k}^t = u_k^b - u_k^t,
\\ T(l_{\overline\alpha}) - l_{\alpha} &= u_{k - 1}^b - u_{k - 1}^t.\end{align*}
Recalling that $\lambda$ being $\nu$-balanced implies \eqref{eq:balancedintervals}, we have 
\begin{equation}
\label{eq:difference_lengths}
\max_{\beta, \beta' \in \A} |\lambda_\beta - \lambda_{\beta'}| < \frac{\nu - \nu^{-1}}{d}.
\end{equation}
Thus, by {the $\delta$-drift assumption} \eqref{eq:delta_drift},
\begin{equation}
\label{eq:drift_explicit}
\begin{aligned} \delta(k)\leq T(l_\alpha) - l_{\overline\alpha} < \nu - \nu^{-1}, \\
\delta \leq T(l_{\overline\alpha}) - l_{\alpha} < \nu - \nu^{-1}.
\end{aligned}
\end{equation}
Then, since 
\begin{equation}
\label{eq:smallness_nu}
\nu - \nu^{-1} < \frac{1}{50d} < \frac{1}{40} \min_{\beta \in \A} \lambda_\beta,
\end{equation}
equations \eqref{eq:difference_lengths} and \eqref{eq:drift_explicit} imply
\[ T^{}(F) \subseteq [l_{\overline\alpha} + \delta(k) + \eta \lambda_\alpha, l_{\overline\alpha} + \tfrac{3}{4}\lambda_\alpha + \nu - \nu^{-1}] \subseteq I_{\overline\alpha},\]
\[ T^{2}(F) \subseteq [l_{\alpha} + \delta + \eta \lambda_\alpha, l_{\alpha} + \tfrac{3}{4}\lambda_\alpha + 2(\nu - \nu^{-1})] \subseteq I_{\alpha},\]
and
\[ T^3(F) \subseteq [l_{\overline\alpha} + \delta + \eta \lambda_\alpha, l_{\overline\alpha} + \tfrac{3}{4}\lambda_\alpha + 3(\nu - \nu^{-1})] \subseteq I_{\overline\alpha} ,\]
\[ T^4(F) \subseteq [l_{\alpha} + \delta + \eta \lambda_\alpha, l_{\alpha} + \tfrac{3}{4}\lambda_\alpha + 4(\nu - \nu^{-1})] \subseteq I_{\alpha},\]
 \[T^5(F) \subseteq [l_{\overline\alpha} + \delta + \eta \lambda_\alpha, l_{\overline\alpha} + \tfrac{3}{4}\lambda_\alpha + 5(\nu - \nu^{-1})] \subseteq I_{\overline\alpha}.\]
This proves the first assertion. 

Notice that by \eqref{eq:difference_lengths}, \eqref{eq:smallness_nu} and since $0 < \delta < \tfrac{1}{10d}$, the previous equations {also imply
 that} \begin{equation}
\label{eq:image_F}
T(F) \subseteq [l_{\overline\alpha} + \delta(k), r_{\overline\alpha} - \delta],
\end{equation}
and that
\begin{equation}
\label{eq:iterates_F}
\begin{aligned}
T^2(F), T^4(F) \subseteq [l_{\alpha} + \delta, r_{\alpha} - \delta], \\ T^3(F), T^5(F) \subseteq [l_{\overline\alpha} + \delta, r_{\overline\alpha} - \delta].
\end{aligned}
\end{equation}
Notice that by definition of $F$, we have $|a_\alpha - l_\alpha| = \eta |I_\alpha|$. We can now prove the last two assertions:

\begin{enumerate}
\item 
If $k \neq d$, it follows from \eqref{eq:image_F} that $T(F) \subseteq [l_{\overline\alpha} + \delta, r_{\overline\alpha} - \delta]$. This, together with \eqref{eq:iterates_F}, prove the second assertion. 

\item
If $k = d$, then $T(l_\alpha) = 0 = l_{\overline\alpha}$ and thus $|T(a_\alpha) - l_{\overline\alpha}| = \eta|I_\alpha|$. By \eqref{eq:iterates_F}, the third assertion follows. 
\end{enumerate}
\end{proof}

\subsection{Good renormalization times}
We now define, in \S~\ref{sec:goodtimesdef}, the renormalization times that will help us produce good base dynamics and good towers above it (see Definition~\ref{def:goodtimes}). We then state, in \S~\ref{sec:Egoodtimes}, a result about their existence for a.e.~IET (see Proposition~\ref{prop:full_measure_cond}). 

\subsubsection{Discontinuities and derivative control}
\label{sec:propgoodtimesdef}
Let us introduce a few definitions used in describing good renormalization times. 
It is rather classical to consider times when the Rauzy-Veech induction is \emph{balanced} in the following sense.
\begin{definition}[Balanced times]\label{def:balancedtime}
Given an IET $T = (\pi, \lambda)$ and $\nu > 0$, we say that $n \geq 0$ is a \emph{$\nu$-balanced time for $T$} if $\lambda^n$ and $q^n$ are $\nu$-balanced, i.e., 
\begin{equation}
\label{eq:balanced_vectors}
\max_{\alpha, \beta \in \A} \frac{\lambda^n_\alpha}{\lambda^n_\beta} < \nu, \qquad \max_{\alpha, \beta \in \A} \frac{q^n_\alpha}{q^n_\beta} < \nu.
\end{equation} 
\end{definition}
\noindent Thus, both the lenght vector $\lambda^n$ and the height vector $q^n$ are $\nu$-balanced in the sense of Definition \ref{def:balance}. 

\smallskip
The next condition guarantees that the location of the discontinuities in Rohlin towers is controlled. 
\begin{definition}[Well-positioned discontinuities]\label{def:well}
Given an IET $T = (\pi, \lambda)$ with a symmetric permutation, we say that $T$ has \emph{well-positioned singularities at time $n$} if $\pi^n = \pi$ and, { for any $\alpha \in \A$ with $1 < \pi_0(\alpha) < d$, there exists an integer $i_\alpha^n$ such that}
\begin{equation}
\label{eq:pos_discontinuities}
\tfrac{q^n_\alpha}{4} < i_\alpha^n < q^n_\alpha \quad \text{and } \quad l_\alpha = T^{i^n_\alpha}(l_\alpha^n). 
\end{equation}
\end{definition}
Combining balanced times with $\delta$-drift for the induced IET (see 
Definition~\ref{def:drift}), we will be able to obtain good rigidity properties 
for the Rohlin towers associated with the induction procedure (see Lemma 
\ref{lem:towers_approaching_left_singularities}). The last ingredient crucial 
for tightness is {the existence of times, which we call \emph{good derivative 
control times}, where the following control of the Birkhoff sums 
of the first derivative (which is not an integrable function) holds (whose 
existence was proved in the work of the last author, see 
\cite{ulcigrai_absence_2011}).}

\begin{definition}[Good {derivative control} times]\label{def:cancelltimes}
Given an IET $T: I \to I$, a cocycle $\varphi \in LSS^{2}(\bigsqcup_{\alpha\in\mathcal A} I_\alpha)$ of the form \eqref{eq:logartihmic_singularities}, and $M > 0$, we say that $n \geq 0$ is a \emph{$M$-good derivative control time} { for $\varphi$} if, for any {$z \in I^n$} and any $0 \leq r {<} q_\alpha^n$,
\[\left|S_r {\varphi'}(z)-\left(\sum_{\alpha \in \A}\frac{C_\alpha^+}{z_{\alpha}^+}+\frac{C_\alpha^-}{z_{\alpha}^-}\right)\right| \leq Mr,\]
where $z_\alpha^+$ and $z_\alpha^-$ are the closest visits of $z$ to the 
{endpoints $l_\alpha$ and $r_\alpha$ respectively,} {among the first $r$ iterates $\{T^i z, 0\leq i<r \}$ of $z$ under $T$}. 
\end{definition}
{ We remark that since ${\varphi'}$ has singularities of type $1/x$, which are not integrable, the conclusion of Birkhoff ergodic theorem does not apply to ${\varphi'}$ and the existence of good derivative control times as the above can only be proved when the logarithmic singularities are \emph{symmetric} and under a (full measure) delicate Diophantine-like condition for the IET (see \cite{ulcigrai_absence_2011} or the survey \cite{Ul:ICM} for further details), while it typically fails when the singularities are asymmetric (see \cite{Ulcigrai_mixing_2007, Rav:mix}). }

\subsubsection{Definition of good renormalization times}\label{sec:goodtimesdef}
{For a.e. IET $T$ with a symmetric permutation,} we will show (see Proposition~\ref{prop:full_measure_cond}) that all of the properties above can be realized simultaneously along a subsequence of its {iterates under Rauzy-Veech induction, which we call \emph{good renormalization times}:}

\begin{definition}[Good renormalization times]\label{def:goodtimes}
{Given an IET $T: I \to I$ with continuity intervals ${(I_\alpha)}_{\alpha\in\mathcal A}$ satisfying the Keane condition, 
and a cocycle $\varphi \in LSS^{2}(\bigsqcup_{\alpha\in\mathcal A} I_\alpha)$,} 
we say that a sequence $(m_n)_{n\in\N}$ is a sequence of {$(\nu, \delta, 
M)$-\emph{good renormalization times} {for $\varphi$ at $\alpha\in 
\mathcal{A}$}, where $\nu>1, \delta>0, M>0$,} if the following holds for any $n 
\in \N$:
\begin{enumerate}
\item \label{cond:balanced_time} $m_n$ is a $\nu$-balanced time for $T$ {(in the sense of Definition~\ref{def:balance})};
\item $T$ has well-positioned discontinuities at time $m_n$ {(in the sense of Definition~\ref{def:well})};
\item $\mathcal{RV}^{m_n}(T)$ has a $\delta$-drift at $\alpha$ {(in the sense of Definition~\ref{def:drift})};
\item \label{cond:good_derivative} $m_n$ is a {$M$-good} derivative 
control { time} {for $\varphi$} {(in the sense 
of Definition~\ref{def:cancelltimes})}.
\end{enumerate}
{ We say that a sequence $(m_n)_{n\in\N}$ is a sequence of \emph{good 
renormalization times} (for $\varphi$ at $\alpha$) if it is a $(\nu, \delta, 
M)$-\emph{good renormalization times} (for $\varphi$ at $\alpha$) for some 
$\nu>1, \delta>0$ and $M>0$.}
\end{definition}

{
In \S~\ref{sc:proof_full_measure_cond}, we will prove that good renormalization times exist for almost every IET with symmetric permutation (see Proposition~\ref{prop:full_measure_cond}). 
We now show how to use good renormalization times to define the rigidity sets $\Xi_n$, which will satisfy the conditions of Proposition \ref{prop:ergodicity_criterion} {(see \S~\ref{sec:goodtimes}).} 

\subsection{Good towers dynamics}}
{Given a Keane IET $T$ with symmetric permutation which admits a sequence of good renormalization times $(m_n)_{n\in\N}$, for every $n\in\N$, applying Lemma \ref{lem:rigid_interval} to each IET $\mathcal{RV}^{m_n}(T)$,} we can obtain, {in each continuity interval of $\mathcal{RV}^{m_n}(T)$, a subinterval with good base dynamics such that its iterates under $T$ are contained in a Rohlin tower and have explicitly controlled distance from the endpoints of $T$}. More precisely, we have the following:

\smallskip
\begin{lemma}[Good tower dynamics]
\label{lem:towers_approaching_left_singularities}
Suppose { $1< \nu< 1 + \tfrac{1}{100d}$ and $0<\delta< \tfrac{1}{10d}$ as in Lemma \ref{lem:rigid_interval}.}
 Let $T = (\pi, \lambda)$ be an IET with a symmetric permutation and let {$(m_j)_{j \in \N}\subseteq \N$ be a time of a sequence
of $(\nu,\delta, M)$-good renormalization times for $\varphi \in 
LSS^{2}(\bigsqcup_{\alpha\in\mathcal A} I_\alpha)$ 
at some $\alpha = \pi_0^{-1}(k)$, with $1 < k \leq d$.} 
 Then for any {$n\in (m_j)_{j \in \N}\subseteq \N$ }
and any $0 < \eta < \frac{1}{4}$, 
the interval 
\begin{equation}\label{def:F}
{F}= F(\eta,\alpha): = [a_\alpha^n, b_\alpha^n] = [l_\alpha^n + \eta |I^n_\alpha|, r_\alpha^n - \tfrac{1}{4{\nu}}|I^n_\alpha|],\end{equation} satisfies the following:

\begin{enumerate}
 \item \label{cond:big_int} $\frac{|{F}|}{|I^n_\alpha|} > \frac{1}{2}.$
 \item \label{cond:nested_images} ${F}, T_n^2({F}),$ $T_n^4({F}) \subseteq I^n_{\alpha}$ and $T_n({F}),$ $T_n^3({F}),$ $T_n^5({F}), \subseteq I^n_{\overline{\alpha}}$.
 \item \label{cond:3} If $k \neq d,$ then $|T^{i_\alpha^n}(a_\alpha^n) - l_\alpha| = \eta|I^n_\alpha|$ for a unique $\tfrac{q^n_\alpha}{4} < i_\alpha^n < q^n_\alpha$ and 
 $$|T^i(a_\alpha^n) - l_\beta | \geq \delta {|I^n|},\ {\text{for}\ 0\leq i\leq {5}(q_\alpha^n + q_{\overline{\alpha}}^n) \ and\ \beta \in \A \ s.\ t.\ (i,\beta) \neq (i_\alpha^n, \alpha) }.$$
 \item \label{cond:4} If $k = d$, then $|T^{q_\alpha^n - 1}(a_\alpha^n) - l_\alpha| = \eta|I_\alpha^n| = |T^{q_\alpha^n}(a_\alpha^n) - l_{\overline{\alpha}}| $ and 
 $$|T^i(a_\alpha) - l_\beta| \geq \delta {|I^n|},\quad {\text{for}\ 0\leq i\leq {5}(q_\alpha^n + q_{\overline{\alpha}}^n),\ \beta \in \A \ s.\ t.\ (i,\beta) \notin \{(q^n_\alpha - 1, \alpha), (q^n_\alpha, \overline{\alpha})\} }.$$
\end{enumerate}
\end{lemma}

\begin{proof}
{ Let $(m_j)_{j \in \N}\subseteq \N$ be a sequence of $(\nu, \delta, M)$-good times and let $n:=m_j$ for some $j\in\N$ be one of such times.} It is well-known that the {endpoints} of a given IET satisfying the Keane condition belong to the orbits under $T$ of the {endpoints of any of the IETs in its Rauzy-Veech induction orbit. Since by Condition \ref{cond:good_derivative} of Definition~\ref{def:goodtimes}, for any $n\in \N$, 
$T$ has well-positioned discontinuities at time $m_n$,} so in particular the permutation of the IET $T_n$ is the same as that of $T$ and for any $\alpha\in\mathcal{A}$ such that $1 < \pi_0(\alpha) < d$, 
\[ l_\alpha = T^{i^n_\alpha}(l_\alpha^n), \quad \text{for a unique } \tfrac{q^n_\alpha}{4} < i_\alpha^n < q^n_\alpha,\] 
while if $\pi_0(\alpha) = d$, then
\[ l_\alpha = T^{q^n_\alpha - 1} (l_\alpha^n), \quad\quad l_{\overline \alpha} = T^{q^n_\alpha}(l_\alpha^n). \]
{Thus, since $(\nu, \delta, M)$-times are in particular $\delta$-good times (by Condition \ref{cond:balanced_time} of Definition~\ref{def:goodtimes})} 
by Lemma~\ref{lem:rigid_interval} applied to each $\mathcal{RV}^{m_n}(T)$ {and the above remarks, the properties stated in the conclusion of the Lemma follow.}
\end{proof}

\subsubsection{Definition of the good rigidity sets}\label{sec:goodtimes}
{Let $T = (\pi, \lambda)$ be a Keane IET with a symmetric permutation and let
 $(m_n)_{n \in \N}\subseteq \N$ be a sequence of $(\nu, \delta, M)$-good 
 rigidity times at $\alpha = \pi_0^{-1}(k)$ for $\varphi \in 
 LSS{^2}(\bigsqcup_{\alpha\in\mathcal A} I_\alpha)$ of the form 
 \eqref{eq:logartihmic_singularities}, where
 $1< \nu< 1 + \tfrac{1}{100d}$ and $0<\delta< \tfrac{1}{10d}$.} 
{For any given parameter $0 < \eta < \tfrac{1}{4}$,} we can define the sets {$\Xi_n=\Xi_n(\eta, \alpha)$ and the heights $h_n=h_n(\eta)$} that we will use as rigidity sets to verify the ergodicity criterion assumptions. {For any $n \in \N$, let
$$
{F^n_\alpha}:= [a^{m_n}_\alpha, b^{m_n}_\alpha]
$$
be the interval given by Lemma \ref{lem:towers_approaching_left_singularities} applied to the time $m_n$ and let {$q^{m_n}_\alpha$ } the height of the Rohlin tower over 
$I^{m_n}_\alpha\supseteq F^n_\alpha$. 
{Denoting by $\lceil x \rceil$ the ceiling of $x$ (i.e., the smallest integer larger than $x$),} define, for any $n \in \N$,
\begin{equation}\label{eq:set_Xi_n}
\Xi_n = { \Xi_n (\eta, \alpha) := \bigsqcup_{i = 0}^{\lceil q_\alpha^{m_n}/4\rceil} T^i(F^{n}_\alpha),}
\end{equation}
\begin{equation}
\label{eq:heights_Xi_n}
h_n { : = q_\alpha^{m_n} + q_{\overline\alpha}^{m_n}}.
\end{equation}
Notice that the sets $\Xi_n$ depend on $\eta$, {since the definition of $F^n_\alpha$ does, see \eqref{def:F}}. {Since} the intervals $T^i(F_\alpha^{n})$, for $i=0,\ldots,h{^{n}}-1$, are pairwise disjoint { (and thus the union in \eqref{eq:set_Xi_n} is disjoint),} { by Condition \ref{cond:big_int} in Lemma \ref{lem:towers_approaching_left_singularities}, we get
\begin{equation}\label{eq: boundonF_nlength}
	\frac{|I^{m_n}_{\alpha}|}{2}\le \textup{Leb}(F^{n}_{\alpha})\le 
	\frac{1}{h{_{n}}}.
\end{equation}}
{ In particular, if \eqref{eq:balanced_vectors} is satisfied, 
then there exists 
a constant $\delta>0$ such that
\[
\textup{Leb}(\Xi_n)>\delta\quad \text{for every }n\in\N.
\]}
\noindent It is not difficult to see that with these definitions, the partial rigidity conditions in Definition \ref{def:partial_rigidity} are satisfied by the sequences $\{\Xi_n\}_{n\in\N}$ and $\{{h_n}\}_{n \in \N}$, as well as by the sequences $\{\Xi_n\}_{n\in\N}$ and {$\{{2 h_n}\}_{n \in \N}$}. {As we shall see in later sections, {selecting the value of $\eta$ appropriately, the sequences $\{\Xi_n\}_{n\in\N}$ and {$\{{2 h_n}\}_{n \in \N}$} will also satisfy the tightness and oscillations assumptions (Conditions \ref{cond:tightness} and \ref{cond:decay_coefficients}) in Proposition \ref{prop:ergodicity_criterion}, thus guaranteeing the ergodicity of the skew-product $T_f$. }

\subsubsection{Estimates of Birkhoff sums along good rigidity sets}\label{sec:estimates}
{
In the rest of this section, we state and prove a result (which will be used in the proof of the main result in the last section, \S~\ref{sec:last}) which provides control for Birkhoff sums of points in the good rigidity sets $\{\Xi_n\}_{n \in \N}$ at times $\{2 h_n\}_n$, where $\{h_n\}_{n \in \N}$ are good rigidity times, namely for Birkhoff sums of the form $S_{2 h_n}\varphi' (x)$, where $x\in \Xi_n$ (see Proposition~\ref{lem:BScontrol} below).

Throughout this section, as before, we always assume that { $1< \nu< 1 + \tfrac{1}{100d}$ and $0<\delta< \tfrac{1}{10d}$ (so that Lemma~\ref{lem:towers_approaching_left_singularities} can be applied) and that $0<\eta<1/4$. 

\begin{proposition}\label{lem:BScontrol}
There exists a constant $D=D(\delta, \nu, M)>0$ such that, for any $\alpha\in \A$ 
and any sequence $\{m_n\}_{n\in \N}$ of $(\delta, \nu, M)$-good times for 
$\varphi$ at 
$\alpha$, for any $n\in\N $ and any $x\in \Xi_n(\alpha,\eta)$ (where $\Xi_n$ are the sets defined in the previous subsection \S~\ref{sec:goodtimes}) we have 
\begin{equation}\label{eq:lowerboud}
\left|S_{2h_n}{\varphi'}(x) - \frac{C^+_{\gamma(\alpha)}}{|x_0-l^{m_n}_\alpha|}
\right|\le D q^{m_n}_{\alpha} ,
\end{equation}
where $x_0\in F^n_\alpha$ is the projection of $x=T^i (x_0)$ to the base of the tower $\Xi_n$, and $\gamma(\alpha) \in \A$ is given by 
\begin{equation}\label{def:gamma}
\gamma(\alpha):= \begin{cases} \alpha & \text{if} \ \alpha=\pi_0^{-1}(k),\ k\neq d, \\ \overline{\alpha} & \text{if}\ \alpha = \pi_0^{-1}(d). \end{cases} 
\end{equation}
\end{proposition}
\noindent As we will show, the term $C^+_{\gamma(\alpha)}/|x_0-l^{m_n}_\alpha|$ in \eqref{eq:lowerboud} is exactly the contribution given by the closest visit to a singularity from the right, which occurs during the first segment of the orbit of $x$ up to the tower and back to the base. 

\smallskip
The rest of this section is devoted to the proof of this proposition.
\begin{proof}
We will split the proof into several steps. Let us first identify the closest visit to singularities. Throughout the proof, we fix $\alpha \in \A$ which we write $\alpha = \pi_0^{-1}(k)$, with $1 < k \leq d$ and fix $x\in \Xi_n(\eta, \alpha)$, which by definition of $\Xi_n$, we can write as 
\begin{equation}\label{def:x} 
x=T^i (x_0), \qquad x_0\in F^\alpha_n, \quad\ 0\leq i<\tfrac{q^{m_n}_\alpha}{4}. 
\end{equation}

\noindent {\it Closest visit.}
If we consider an orbit of $x$ of length at least $q^{m_n}_\alpha - i {+1}$ (i.e., all the way to the top of the tower and back to the base once), we claim that the closest visit to a singularity to the left has distance
\begin{equation}\label{eq:closest}
\left| x_0-l^{m_n}_\alpha \right| = \min_{0\leq \ell < q^{m_n}_\alpha -i+1} \text{dist} \left( T^\ell (x) , \{ l_\beta, \beta \in \A\}\right).
\end{equation}
To see this, let $i^{m_n}_\alpha $ be the height of the floor of the tower over $I^{m_n}_\alpha$ whose left endpoint is a singularity, i.e. $l_\alpha = T^{i^{m_n}_\alpha}(l_\alpha^{m_n})$. Since by the assumption that discontinuities are well-positioned (see $2$ in Definition~\ref{def:goodtimes}), the orbit of $x$ visits such floor, and the closest distance to a singularity to the left is given by 
\begin{align*}
 \left| T^{i_\alpha^n -i } (x)- l_\alpha\right| & = \left| T^{i_\alpha^n} (x_0)- l_\alpha\right| = \left| x_0- l^{m_n}_\alpha \right|, \qquad \text{if}\ k \neq d, \\
 \left| T^{q_\alpha^{m_n}-i} (x)- l_{\overline{\alpha}} \right|& = \left| T^{q_\alpha^{m_n}} (x_0)- l_{\overline{\alpha}}\right| = \left| x_0-l^{m_n}_\alpha\right| , \qquad \text{if}\ k = d,
\end{align*}
(where the equalities follow from the definition \eqref{def:x} of $x_0$ and by remarking that $T^{i_\alpha^n}$ acts as an isometry on the base of the tower, respectively). These closest visits produce a contribution of the form ${C_{\gamma(\alpha)}^+}/{|x_0-l_{\overline{\alpha}}|}$ in the Birkhoff sums of $\varphi'$, where $\gamma(\alpha) $ is as in \eqref{def:gamma}. 


\smallskip
\noindent {\it Decomposition.}
To control the Birkhoff sum $S_{2 h_n}\varphi' (x)$ where $x\in \Xi_n$, 
in view of the good tower dynamics given by Lemma \ref{lem:towers_approaching_left_singularities}, we can split the orbit of $x$ of length $2h_n=2q_\alpha^{m_n}+ 2q_{\overline{\alpha}}^{m_n}$ into $5$ orbit segments, one of length $q_\alpha^{m_n}-i$ (up to the top of the tower), three of length equal to the height of a tower and finally the last segment of length $i$ as follows: if we denote $x_{i+1}:= T^{(m_n)} (x_i)$, for $i\in\N$, then
\begin{equation}\label{eq:segments}
S_{2h_n} \varphi' (x)= S_{ q_\alpha^{m_n}-i} \varphi' (x)+ S_{q_{\overline{\alpha}}^{m_n} } \varphi' (x_1) + S_{ q_\alpha^{m_n}} (x_2) \varphi' (x_2) + S_{ q_{\overline{\alpha}}^{m_n} } \varphi' (x_3)+ S_{i} \varphi' (x_4).
\end{equation}
The three middle terms will be estimated by applying the $M$-good derivative control in Assumption \ref{cond:good_derivative} of Definition \ref{def:goodtimes} directly. The next step provides an estimate which will allow us to combine the first and last term and compare them to the Birkhoff sum $S_{q_{\alpha}^{m_n}}\varphi'(x_0)$. We remark that we only obtain an upper bound; the first and last term could be estimated separately, seeing the first as the difference of two terms to which the $M$-good derivative control applies, but this does not provide the bound \eqref{eq:lowerboud} since it would \emph{overcount} (i.e., count twice) the contribution of the closest visit.

\smallskip
\noindent {\it Control of variation via the second derivative.} We claim that there exists $K=K(\nu,\delta, M)>0$ such that, for any $n \in \N$ and any $x, y \in F^n_\alpha,$ 
\begin{equation}\label{eq:variation}
|S_{k}{\varphi'}(x) -S_{k}{\varphi'}(y) |\le K q^{m_n}_{\alpha}, \qquad \text{for \ all}\ 0\leq k\leq \frac{q^{m_n}_{\alpha}}{4}.
\end{equation}
(This estimate is used in the next step to compare the first and last term of the RHS of \eqref{eq:segments} to $S_{q_{\alpha}^{m_n}}\varphi'(x_0)$). By the mean value theorem, 
 there exists $z\in J:= (x, y) \subseteq F^n_\alpha$ such that
\begin{equation}\label{eq:meanvaluestep1}
|S_{k}{\varphi'}(x) -S_{k}{\varphi'}(y) |\le \left| S_{k}{\varphi''}(z)\right| |J|\leq \left| S_{k}{\varphi''}(z)\right| \frac{1}{q_{\alpha}^n}.
\end{equation}
Let us now estimate the Birkhoff sum $S_{k}{\varphi''}(z)$. Assuming that $|I|=1$ and since $\varphi$ is of the form \eqref{eq:logartihmic_singularities}, we have
\[
\varphi ''(x) = g''_\varphi (x) + \sum_{\beta \in \A} \frac{ C_\beta^+}{\left\{ (x - l_\beta)^2\right\}} 
+ \sum_{\beta \in \A} \frac{ C_\beta^-}{ \left\{ (r_\beta - x)^2 \right\} }.\]
Notice that the Birkhoff sums $S_k g''_\varphi (x)$ are trivially bounded by {$k \|g''_\varphi\|_\infty $,} for any $k \in \N$. Since we are assuming that $0\leq i <q^{m_n}_\alpha$ and (by the Definition~\ref{def:goodtimes} of good renormalization times) the discontinuities of $T$ are well positioned at each time $m_n$ (in the sense of Definition~\ref{def:well}), every point in the orbit in the orbit $\{ z, T(z), \dots T^{k-1}(z)\} $ has distance at least $|I^{m_n}_\alpha|$ from any singularity of $f$. Moreover, the points in this orbit segment (since they belong to distinct floors of the tower over $I^n_\alpha$) have minimum spacing (i.e., $\min_{i,j} |T^i(z)-T^j(z)|$ for $0\leq i, j< k$) also bounded below by $I^{m_n}_\alpha$. Thus, rearranging them from the closest to the endpoints, using monotonicity and $|I_\alpha^{m_n}|q_{\alpha}^{m_n}\leq 1$, we get, for each $\beta \in \A$, 
$$
\left| \sum_{j=0}^{k-1}\frac{C_\beta^+ }{(T^j (z) - l_\beta)^2}\right| \leq
\sum_{j=1}^{k-1}\frac{C_\beta^+}{j^2 |I^{m_n}_\beta|^2} \leq
\frac{\pi^2}{2} C_\beta^+ (q^{m_n}_\beta)^2 . 
$$ 
A similar estimate holds for the sums corresponding to $C_\beta^-$ and $r_\beta$. Adding up all these estimates and recalling \eqref{eq:meanvaluestep1}, we get
$$|S_{k}{\varphi'}(x) -S_{k}{\varphi'}(y) |\le \left| S_{k}{\varphi''}(z)\right| \frac{1}{q_{\alpha}^{m_n}} \leq 2d C_{max} \frac{\pi^2}{2} q_{\alpha}^{m_n}, $$
where $C_{max}:= \max \{ C_\beta^\pm, \beta \in \A\}$, which proves \eqref{eq:variation} for $K:= 2dC_{\max }\pi^2/2$. 
\smallskip

\smallskip
\noindent {\it Final estimates.} Let us now estimate all terms of \eqref{eq:segments}. The sum of the first and the last term of RHS of \eqref{eq:segments} can be compared to the Birkhoff sum $S_{q_\alpha^{m_n}}\varphi (x_0)$ of the base point $x_0 \in F_\alpha^n$ using the previous step. Indeed, since $S_{q_\alpha^{m_n}}\varphi (x_0) = S_{i}\varphi (x_0 ) + S_{ q_\alpha^{m_n}-i} \varphi (x)$, we have
$$
\left| S_{ q_\alpha^{m_n}-i} \varphi (x)+ S_{i} \varphi \left(x_4\right)- S_{q_\alpha^{m_n}}\varphi (x_0) \right| 
= \left| S_{i} \varphi (x_4) - S_{i}\varphi (x_0 ) \right|\leq 
 C q_\alpha^{m_n}.
$$
Therefore, from this estimate and \eqref{eq:segments}, we get
\begin{equation}\label{eq:4terms}
\left| S_{2h_n} \varphi (x)- \left( S_{ q_\alpha^{m_n}} \varphi (x_0)+ 
 S_{ q_
{\overline{\alpha}}^{m_n} } \varphi (x_1) + S_{ q_\alpha^{m_n}-i} \varphi (x_2)
+ S_{ q_{\overline{\alpha}}^{m_n} } \varphi (x_3)\right) \right| \leq C q_\alpha^{m_n}.
\end{equation}
Let us now apply the $M$-good derivative control (i.e. Assumption \ref{cond:good_derivative} of Definition~\ref{def:goodtimes}) to each of the $4$ Birkhoff sums (of height $q^{m_n}_\alpha$ or $q^{m_n}_{\overline{\alpha}}$) in \eqref{eq:4terms}. Setting aside what is the contribution from the first closest visit \eqref{eq:closest} (see Step $1$), {by the good tower dynamics described by Lemma \ref{lem:towers_approaching_left_singularities} (see in particular Conditions \ref{cond:3} or \ref{cond:4}, according to which $\alpha\in \A$ is being considered) in Lemma~\ref{lem:towers_approaching_left_singularities},} we can control all other closest visits to a singularity from the left in each of the $4$ orbit segments by l by $\delta |I^n|$. Thus 

\begin{equation}
\left| S_{2h_n} \varphi (x) - \frac{C^+_{\gamma(\alpha)}}{|x_0-l^{m_n}_\alpha |} \right| \leq
 K q_\alpha^{m_n} + M (2 q_\alpha^{m_n}+ 2 q^{m_n}_{\overline{\alpha}} )
 +
 \frac{(8d-1) C^\pm_{max}}{\delta|I^{m_n}|},
\end{equation}
 where $C^\pm_{max}:= \max \{|C_\alpha^\pm|, \ \alpha\in\A \}$ and $K$ is 
 given by the second step. 
Using balance (which gives that $q_{\overline{\alpha}}^{m_n}\leq \nu 
q_\alpha^{m_n}$, as well as that $|I^{m_n}| q_\alpha^{m_n}\geq |I^{m_n} | 
\min_\beta q_\beta/\nu \geq 1 /d\nu $), this proves the lemma by setting 
$D=D(\eta, \delta, M):=K + M (2+2\nu) + d\nu (8d-1)C^\pm_{max} /\delta$. 
\end{proof}}}

\subsection{Cancellations and involution fixed points}\label{sec:cancellations}
In this section, { we exploit the symmetry of the permutation and the oddness assumption on the cocycle to show that there are cancellations of Birkhoff sums (which will later be crucial to verify the tightness assumption needed to prove ergodicity). The goal of this section is to prove the following:

\begin{proposition}[Approximate zeros of Birkhoff sums]\label{lem: zeroes}
Let $T=(\pi,\lambda)$ be an IET with a symmetric permutation $\pi$ and exchanged intervals $(I_\alpha)_{\alpha\in\mathcal{A}}$ and let $\varphi \in LSS^2(\bigsqcup_{\alpha \in \A} I_{\alpha})$. Assume $\{m_n\}_{n\in\N}$ is a sequence of $(\eta,\delta,M)$-good rigidity times for $\varphi $, for $\delta, \nu$ as in Lemma~\ref{lem:rigid_interval}. 
Then, for $n\in \N$ sufficiently large and any $\alpha\in\mathcal A$, there exists a point { $x^n_{\alpha}$ in the Rohlin tower $\bigsqcup_{i=0}^{q^{m_n}_{\alpha}-1}T^i\big(I^{m_n}_{\alpha}\big)$} satisfying
	\begin{equation}\label{eq" known_bounded_sum}
	|S_{2h_n}\varphi ({x^{n}_{\alpha}})|\le \frac{1}{1000},\qquad \text{where} \ h_n= q^{m_n}_\alpha + q^{m_n}_{\overline{\alpha}}.
	\end{equation}
\end{proposition}
\noindent To prove the proposition, it will be crucial to control the location of the midpoints, or \emph{centers}, of each interval $I_\alpha$, namely 
$$
c_{\alpha}:= \frac{r_\alpha-l_\alpha}{2}, \qquad \text{where}\ I_\alpha:= [l_\alpha, r_\alpha), \quad \alpha \in \A,
$$
and of the midpoint $c_{1/2}$ of $I$, 
relative to the partitions produced by renormalization. Indeed, as shown in the following lemma, the midpoints have good cancellations properties.}

\begin{lemma}	\label{cor:bounded_sums}
	Let $T: I \to I$ be an IET with a symmetric permutation and {let $\varphi: I \to \R$ be a cocycle odd with respect to each interval of continuity, namely, a cocycle which verifies} \eqref{eq:oddly_symmetric}. Then, for every $n \geq 1$,
	\[ S_{2n} \varphi (T^{-n}(x_0)) = 0,\]
	where $x_0$ 
	is the middle point of $I$, and
	\[ S_{2n} \varphi (T^{-n}(x)) = - \varphi (T^n(x)),\]
{ when $x=c_\alpha$, for any $\alpha\in\mathcal{A}$, where $c_\alpha$ is the middle point the continuity intervals $I_\alpha$ for $T$.}
\end{lemma}

\medskip

{\noindent {\it Location of the centers.} Given an IET $T: I \to I$ and satisfying Keane's condition (so that the inducing intervals $I^n_{\alpha}$ are defined for any $n\in \mathbb{N}$), let us denote by $c^n_{\alpha}$ the centers 
of the inducing subintervals $I^n_{\alpha}$, for $\alpha \in \A$ and $n\in 
\mathbb{N}$ and by $c^n_{1/2}:|I^n|/2$ the midpoint of $I^n$.}
The following lemma { guarantees that at the $n$-th step of induction, each center is a vertical projection on the base of either one of the centers of the initial continuity intervals or of $I$.

\begin{lemma}[Location of the centers]\label{lm: position_of_centers}
	Let $T=(\pi,\lambda)$ be an IET satisfying Keane's condition with symmetric permutation $\pi$ and $|\lambda|=1$. Fix any $n\in \N$ such that $\pi^n=\pi$ and denote {by $|I^n|=\sum_{\alpha\in\mathcal{A}} \lambda^n_\alpha$} the length of the $n$-th inducing subinterval {$I_n$}.
{
	 Then, for any $\alpha \in \mathcal{A}\cup \{ 1/2\}$ there exist $s_\alpha^n \in \Z$ and $\alpha^-,\beta_1,\beta_2\in\mathcal{A}$, verifying
$$-q^n_{\alpha^-}\le {s}_\alpha^n<q^n_{\alpha},\ \ \text{for \ }\alpha\in\mathcal{A}\quad \text{and}\quad q^n_{\beta_2}\le {s}_{1/2}^n<q^n_{\beta_1},$$
 such that}
	\[
		\left\{c_{\alpha}\mid \alpha\in\mathcal A \cup \{ \tfrac{1}{2}\} \right\} = \left\{T^{{s}_\alpha^n}\big(c_{\alpha}^n\big)\,\big|\, \alpha\in\mathcal A \cup \{ \tfrac{1}{2} \}\right\}. 
	\]
	{	Furthermore, for infinitely many $n$, 
\begin{equation}\label{eq:positive}{0< s}_\alpha^n<q^n_{\alpha}, \ \text{for\ all}\ \alpha \in \mathcal{A}, \qquad \ 0{<s}_{1/2}^n<q^n_{\beta_1}.
\end{equation}}
	\end{lemma} 

\begin{proof}
Let $(\pi,\lambda,\tau)$ be a translation structure on a surface $M$ such that $(\pi$, $\lambda)$ is as in the lemma statement. 
Abusing the notation, let us denote by $\mathcal I: M\to M$ the hyperelliptic involution on $M$, which can be geometrically seen as a rotation by $180^\circ$ on the polygonal representation of $(\pi,\lambda,\tau)$ (see \S~\ref{sec:RV}). The map $\mathcal I$ has exactly $d+1$ {fixed} points: $d$ of these points are the centers of the sides of the polygon $(\pi,\lambda,\tau)$, namely the points
\[
C_{\alpha}:=\sum_{i<\pi_0(\alpha)}(\lambda_{\pi_0^{-1}(i)}+i\tau_{\pi_0^{-1}(i)})\ +\ \frac{1}{2}(\lambda_{\alpha}+i\tau_{\alpha}), \qquad \alpha \in \A.
\]
The last { of the points fixed by} $\mathcal I$ is the center of the polygon, namely
\[
C_{1/2}=\frac{1}{2}\sum_{\alpha\in\mathcal A}(\lambda_{\alpha}+i\tau_{\alpha}).
\]

Pick the classical Poincaré section $I$ for the vertical translation flow on $(\pi,\lambda,\tau)$, a segment of length one whose left endpoint is $0$. Then $T=(\pi,\lambda)$ is the first return map of the translation flow. Note that, for every $\alpha\in\mathcal A$, we have
\begin{equation}\label{eq: centerrelation}
c_{\alpha}=C_{\alpha}+v_{\alpha},\ \text{where}\ v_{\alpha}:=-i\left(\frac{1}{2}\tau_{\alpha}+\sum_{i<\pi_0(\alpha)}\tau_{\pi_0^{-1}(i)}\right),
\end{equation}
and
\begin{equation}\label{eq: 1/2relation}
\frac{1}{2} = C_{1/2}+v_{1/2},\ \text{where}\ v_{1/2}:=-i\left(\frac{1}{2}\sum_{\alpha\in\mathcal A}\tau_{\alpha}\right).
\end{equation}

Consider now 
 { a step $n\geq 1$ of the extended Rauzy-Veech induction $(\pi^n,\lambda^n,\tau^n)=\mathcal R^n(\pi,\lambda,\tau)$ obtained such that 
 $\pi = \pi^n$.}{ Since $\pi$ is symmetric, one can see that the polygonal representation of $(\pi^n,\lambda^n,\tau^n)$ is a polygon without self-intersections. Let $q^n_{\alpha}$ be the} first return time of interval $I^{n}_{\alpha}$ to $I^n$ under the vertical translation flow. Furthermore, $\pi^n=\pi$ being symmetric also implies that
an involution $\mathcal I$ is given the rotation by $\pi$. Moreover, the centers of all sides $C^n_{\alpha}$ as well as the center of polygon $C^n_{1/2}$ are again centers of involution $\mathcal I$, that is,
\[
\{C_{\alpha}\mid \alpha\in\mathcal A \}\cup\{C_{1/2}\}=\{C_{\alpha}^n\mid \alpha\in\mathcal A \}\cup\{C^n_{1/2}\}.
\]
Thus, for every $\alpha\in\mathcal A\cup\{\tfrac{1}{2}\}$, there exists $\beta=\beta(\alpha)\in\mathcal A\cup\{\tfrac{1}{2}\}$ such that $C_{\alpha}=C_{\beta}^n$. We will now discuss three possible cases depending on $\beta$ (Case A, B, and C below).

	\begin{figure}[h!]
	\centering
	\captionsetup{width=.8\linewidth}
		\includegraphics[scale=0.6]{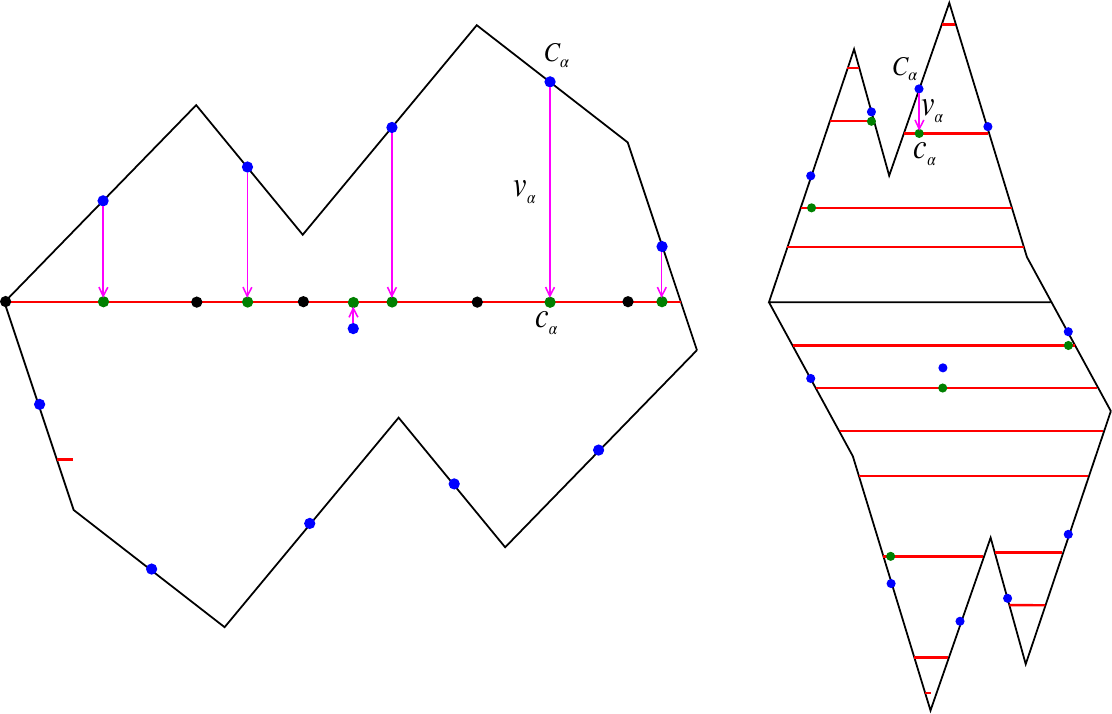}
	\caption{\small \label{fig:polygons} An example of a hyperelliptic surface with centers of involution $C_{\alpha}$ (blue), the corresponding centers of intervals $c_{\alpha}$, the point $\tfrac{1}{2}$ (green), and the vectors $v_{\alpha}$ connecting those points. On the right-hand side, the same surface after $n$ steps of (extended) Rauzy-Veech induction, together with the identification of the special points $C_\alpha$ and $c_{\alpha}$.}
\end{figure}

 { Before we discuss the cases, let us remark that in the polygonal 
 representation of the triple $(\pi^n,\lambda^n,\tau^n)$, which is obtained by 
 cutting and pasting the initial polygonal representation $(\pi,\lambda,\tau)$, 
 the interval $I^0$ (which is a trajectory of the horizontal flow) 
 appears as a union of horizontal intervals contained in the polygon (as shown in Figure~\ref{fig:polygons}, on the right). Furthermore, vertical distances between any two of these intervals are return times of $I^n$ to itself (first return times when no other intervals are crossed in between).} 
 
\smallskip \noindent
\textbf{Case A.} Fix $\alpha\in \mathcal A$ and assume first that the corresponding $\beta$ is an element of $\A\setminus\{\pi_0^{-1}(d)\}$. Then 
\[
C_{\beta}^n=\sum_{i<\pi_0(\alpha)}(\lambda^n_{\pi_0^{-1}(i)}+i\tau^n_{\pi_0^{-1}(i)})\ +\ \frac{1}{2}(\lambda^n_{\alpha}+i\tau^n_{\alpha}).
\]
In particular, $C_{\beta}^n$ {belongs to the forward vertical flow trajectory of $c^n_{\beta}$ of lenght {$q^n_{\beta}$}, where $q^n_\beta$ is the first return time to $I^n_\beta$. Therefore, in the polygonal representation $(\pi^n, \lambda^n,\tau^n)$, $C_{\beta}^n$ lies on a segment of vertical trajectory between two successive intersections with the horizontal intervals representing $I^n$ in the polygon (as shown in Figure~\ref{fig:polygons}). 
Therefore, by \eqref{eq: centerrelation} (and recalling that $\beta=\beta(\alpha) $ was defined so that $C_\alpha=C^n_\beta$),} we have
\[
c_\alpha=C_{\alpha}+v_{\alpha}=C^n_{\beta}+v_\alpha.
\]
Since, for every $\alpha\in\mathcal A\cup\{\tfrac{1}{2}\}$, the point $C_{\alpha}$ lies on a vertical orbit segment which connects the point $c_{\alpha}$ with either $T(c_{\alpha})$ or $T^{-1}(c_{\alpha})$, it has to belong to a floor of the tower $\bigcup_{i=0}^{q_{\beta}^n-1} I^{n}_{\beta}$ immedietaly above or below $C^n_{\beta}$, as seen on the polygon. Thus, there exists $0\le s^n_{\beta}<q_{\alpha}^n$ such that $T^{s^n_{\beta}}\big(c_{\beta}^n\big)=c_{\alpha}$.

\smallskip \noindent
\textbf{Case B.} Assume now that $\beta=\pi_0^{-1}(d)$. If $\sum_{i<d}\tau^n_{\pi_0^{-1}(i)} + \frac{1}{2}\tau^n_{\pi_0^{-1}(d)}>0$, then the procedure is identical to the one in Case A. If $\sum_{i<d}\tau^n_{\pi_0^{-1}(i)} + \frac{1}{2}\tau^n_{\pi_0^{-1}(d)}<0$, then $C_{\pi_0^{-1}(d)}^n$ lies in the backward translation orbit of $c_{\pi_0^{-1}(d)}^n$. To find $s^n_{\pi_0^{-1}(d)}$ it is enough to proceed analogously as in Case A, but using the tower
 $\bigcup_{i=0}^{q_{\beta}-1} I^{n}_{\beta}$, where $\beta=\overline{\pi_0^{-1}(d)}$. In this case, one finds that 
 $$
 -q^n_{\pi_0^{-1}(d)} \leq s^n_\alpha \leq 0,
 $$
which yields the desired time if we set $\alpha^-:=\pi_0^{-1}(d)$.

\smallskip \noindent
\textbf{Case C.} Assume that $\beta=\tfrac{1}{2}$. Note that 
\[
C_{1/2}^n=\frac{1}{2}\sum_{\alpha\in\mathcal A}(\lambda^n_{\alpha}+i\tau^n_{\alpha}).
\]
There are two subcases corresponding to the sign of $\sum_{\alpha\in\mathcal A}\tau^n_{\alpha}$. If it is positive, then the point $C^n_{1/2}$ is in the forward translation orbit of $|I^n|/2$, and one proceeds as in Case A using \eqref{eq: 1/2relation}. Similarly, if it is negative, the point $C^n_{1/2}$ is in the backward translation orbit of $|I^n|/2$, and the conclusion follows analogously to Case B, {(setting $\beta_1,\beta_2$ to be respectively the letters of $\mathcal{A}$ such that ${|I^n|}/{2}\in I^n_{\beta_1}$ and ${|I^n|}/{2} \in T_n^{-1}(I^n_{\beta_2})$).} 

\smallskip
{

	Note finally that if $n$ is such that $\tau^n$ verifies
		\begin{equation}\label{eq:tauchoice}
	\sum_{\alpha\in\mathcal A}\tau^n_{\alpha}>0>\sum_{\alpha\in\mathcal A}\tau^n_{\alpha}-\tfrac{1}{2}\tau^n_{\pi_1^{-1}(d)},
	\end{equation}
 from the previous proof we can conclude that} all $s^n_{\alpha}$, $\alpha\in\mathcal{A}$ as well as $s^n_{1/2}$ are positive. Indeed, $s_\alpha^n$ are always positive in Case A, and under these assumptions also in Case B and Case C,
 the point {$C^n_{\beta}$} {belongs to the segment of the forward (rather than backward) orbit of the point $c^n_{\alpha}$ (where $\alpha=\sigma(\beta)$) consisting of the iterates up to the first return to the interval $I^n$.
 
 It now follows from ergodicity of the normalized extended Rauzy-Veech induction $\tilde{\mathcal{R}}$ that, for almost every choice of the suspension datum $\tau$, there exists infinitely many $n$ such that \eqref{eq:tauchoice} holds (this can be seen for example first choosing a specific parameter $\tau_0$ such that \eqref{eq:tauchoice} holds and then, 
 since the cocycle is locally constant, building an open set $U_0$ in the parameter space $\tilde{\mathcal{M}}$ around $\tau_0$ such that triple in $U_0$ has suspension datum that satisfies \eqref{eq:tauchoice}. Then, it suffices to choose $\tau$ so the orbit of $(\pi,\lambda, \tau)$ is recurrent to $U_0$). Choosing one of such (recurrent) $\tau$ for the proof, the second part of the lemma follows.
}
 \end{proof}

\begin{remark}\label{rem:onlyforward}
The proof shows that for any triple $(\pi,\lambda,\tau)$ giving a polygonal representation, and any $n\in \N$ such that the suspension datum $\tau^n$ of the iterate $\mathcal{R}^n (\pi,\lambda,\tau)$ satisfies \eqref{eq:tauchoice}, the stronger conclusion that all $s_{\alpha}^n$ are positive, namely \eqref{eq:positive}, holds.
\end{remark}

\medskip
{ \noindent {\it Cancellations.}
We can now use the oddness assumption on the cocycle, together with the location of the centers, to show cancellations. For this purpose, we will use the following consequence of the following result (first proved in \cite{chaika_singularity_2021}), whose short proof we recall for completeness.

\begin{lemma}[Lemma 4.6 in \cite{chaika_singularity_2021}]
	\label{lem:value0}
	Let $T:I \to I$ be an IET with a symmetric permutation, and let $\varphi: I \to \R$ verifying \eqref{eq:oddly_symmetric}.
	Then, for every $n \geq 1$ and any $x \in I$,
	\[ S_n\varphi \left(T^{-n}(\mathcal{I}_I(x))\right) = -S_n \varphi (x).\]
\end{lemma}
\begin{proof}
Recall that symmetry of $T$ implies that $T\circ \mathcal{I}_I=\mathcal{I}_I\circ T^{-1}$ where $\mathcal{I}_I(x)=|I|-x$. 
 Since $\varphi $ is odd with respect to each continuity interval and by Lemma~\ref{lem:odd_cocycle}, $$\varphi \circ T^{-1}\circ \mathcal{I}_I=-\phi.$$ Thus 
\begin{align*}
S_n\varphi(T^{-n}(\mathcal{I}_I(x)))&=\sum_{0\leq i<n}\varphi (T^{i-n}(\mathcal{I}_I(x)))=\sum_{0\leq i<n}\varphi (T^{-1-i}(\mathcal{I}_I(x)))\\
&=\sum_{0\leq i<n}\varphi (T^{-1}(\mathcal{I}_I(T^i(x))))
=-\sum_{0\leq i<n}\varphi (T^i(x))=-S_n\varphi(x),
\end{align*}
	which is the desired equality.
\end{proof}}
\noindent {Combining the lemma above with the location of the centers Lemma~\ref{lm: position_of_centers}, we can now prove Lemma~\ref{cor:bounded_sums} and then Proposition~\ref{lem: zeroes}.}
\begin{proof}[Proof of Lemma~\ref{cor:bounded_sums}]
	Let $n \geq 1$. Consider first the middle point $x_0:=|I|/2$ of $I$. Notice that $\mathcal{I}_I(x_0) = x_0$, so it follows from Lemma \ref{lem:value0} that $$ S_{2n} \varphi (T^{-n}(x_0)) = S_{n} \varphi (T^{-n}(x_0)) + S_{n} \varphi (x_0) = 0.$$
	Consider now the case when $x=c_\alpha$, for some $\alpha\in\mathcal{A}$. Since $T$ has a symmetric permutation, it follows that $T(c_\alpha) = \mathcal{I}_I(c_\alpha)$. Thus, by Lemma~\ref{lem:value0}, 
	\[S_{n + 1} \varphi (T^{-n}(c_\alpha)) + S_{n + 1} \varphi (c_\alpha) = 0.\]
	Recall that, by \eqref{eq:oddly_symmetric}, $\varphi $ is odd on each interval of continuity of $T$ and thus $\varphi (c_\alpha) = 0$. Hence
	\begin{align*}
	S_{2n} \varphi (T^{-n}(c_\alpha)) &= S_{n + 1} \varphi (T^{-n}(c_\alpha)) - \varphi (c_\alpha) + S_n \varphi (c_\alpha) \\
	& = S_{n + 1} \varphi (T^{-n}(c_\alpha)) + S_{n+1} \varphi (c_\alpha) - \varphi (T^n(c_\alpha)) 
	= - \varphi (T^n(c_\alpha)),
	\end{align*}
	which also proves the second equality.
\end{proof}

\begin{proof}[Proof of Proposition~\ref{lem: zeroes}] 
{
In view of	Lemma \ref{lm: position_of_centers} and Remark~\ref{rem:onlyforward}, for any $\alpha \in \mathcal{A}$, there exists $\beta=\beta(\alpha) \in \mathcal{A}\cup \{\tfrac{1}{2}\}$ such that the midpoint $c_{\beta}$ belongs to the Rohlin tower over $I^{m_n}_\alpha$ and, more precisely, there exists $-q^{m_n}_{-\alpha} \leq s_{\alpha}^{m_n}< q^{m_n}_\alpha$ such that $c_\beta= T^{ s_{\alpha}^{m_n} }(c_\alpha^n)$. We assume WLOG that $s_{\alpha}^{m_n} > 0$. We will show that
\begin{equation}\label{def:xalphan}
x_\alpha^n:= T^{ s_{\alpha}^{m_n}-h_n }(c_\alpha^n)= T^{-h_n}( c_{\beta}) , \end{equation}
verifies the desired properties.
 
If $\beta=\tfrac{1}{2}$, i.e., $c_{1/2}=|I|/2$ belongs to a floor of the tower over $I^n_\alpha$, the expression on the LHS is just equal to 0 {by the first part of Lemma~\ref{cor:bounded_sums}. 
For any other $\alpha \in \mathcal{A}$, Lemma \ref{lm: position_of_centers}, together with Remark~\ref{rem:onlyforward}, gives that there exists $\beta=\beta(\alpha)\in \mathcal{A}$ such that 
$$c_\beta= T^{s^n_\alpha} (c^n_\alpha), \qquad \text{for\ some}\ {0} \leq s^n_\alpha \leq q_\alpha^n.$$}
Thus, by the definition \eqref{def:xalphan} of $x^n_\alpha$ and by the second part of Lemma~\ref{cor:bounded_sums} (taking $x$ to be $c_\beta$), we have
\begin{equation}\label{eq:semizero}
S_{2h^n}\varphi (x^{n}_{\alpha}) = S_{2h^n}\varphi \left(T^{-h_n} (c_{\beta})\right) = -\varphi \left(T^{h_n}( c_\beta)\right).
\end{equation}
\noindent { Recalling that $h_n= q^{m_n}_\alpha+q^{m_n}_{\overline{\alpha} }$, by the good towers dynamics described by Lemma~\ref{lem:towers_approaching_left_singularities} (since $c_\beta$ is a midpoint and hence belongs to the interval $F^n_\alpha$ which satisties, in particular, Condition \ref{cond:nested_images} of Lemma~\ref{lem:towers_approaching_left_singularities}), $T^{h_n} (c_\beta)$ belongs to the same floor of the Rohlin tower 
over $I^{m_n}_{\alpha}$ than $c_\beta$ (namely, they both belong to $T^{{ s^n_\alpha}} I^{m_n}_{\alpha}$). 
Thus, as $n\to \infty$, $T^{h_n} (c_\beta)$ converges to $c_\beta$. Since 
$\varphi$ is locally continuous, it follows from \eqref{eq:semizero} that 
\[
\lim_{n\to\infty}S_{2h^n} \varphi (x_\alpha^n) = - \lim_{n\to\infty} \varphi \left(T^{h_n}( c_\beta)\right) = -\varphi(c_\beta)=0,
\]
where the last equality follows from the assumption that $\varphi$ is odd on each interval (see Definition~\ref{def:odd}, that implies, since $\mathcal{I}_\alpha (c_\beta)=c_\beta$, that $c_\beta$ is a zero of $\varphi$). 
}}
\end{proof}

\subsection{Full measure of existence of good renormalization times}\label{sec:Egoodtimes}
This section proves that good renormalization times exist for almost every symmetric IET. The main result is the following.
\begin{proposition}[Existence of good renormalization times]
\label{prop:full_measure_cond}
{ For any $1< \nu< 1 + \tfrac{1}{100d}$ and $0<\delta< \tfrac{1}{10d}$,} for 
a.e. IET $T = (\pi, \lambda)$ with a symmetric permutation, any $\varphi \in 
LSS^2(\bigsqcup_{\alpha\in\mathcal A} I_\alpha)$ of the form 
\eqref{eq:logartihmic_singularities}, and any $\alpha \in \A$, there exist { 
$M>0$ and $(m_n)_{n\in \N} \subseteq \N$ such that $(m_n)_{n\in \N}$ is a 
sequence of $(\nu, \delta, M)$-good renormalization times} {for $\varphi$ at 
$\alpha$} (in the sense of Definition~\ref{def:goodtimes}) and, for every 
$\alpha \in \mathcal{A}\cup \{ \tfrac{1}{2}\},$ there exists $\beta\in \A\cup\{\tfrac{1}{2}\}$ 
such that $T^{s_\alpha^n}(c_{\alpha}^n) = c_\beta$, for some $ 0\leq s^n_\alpha 
<q^{m_n}_\alpha$ (i.e.,\eqref{eq:positive} and the stronger conclusion of 
Lemma~\ref{lm: position_of_centers} hold).
\end{proposition}
The proposition above is proved in \S~\ref{sc:proof_full_measure_cond}, after recalling a result from \cite{ulcigrai_absence_2011}, which will be crucial for the $M$-good derivative control at good times.

\subsubsection{Control of $M$-good derivative times via Rauzy-Veech induction}
\label{sc:goodderivative}
The last key ingredient needed for the proof is the {existence of $M$-good derivative control times}. The third author showed in \cite{ulcigrai_absence_2011} that it is possible to obtain $M$-good derivative control times {by considering return times to suitably defined compact sets in the parameter space $ \tilde{\mathcal M}$ of the extended Rauzy-Veech induction. Given a compact subset $K\subseteq \tilde{\mathcal{M}}$, let $A_K$ be the accelerated cocycle corresponding to returns to $K$ (see \cite{ulcigrai_absence_2011} for definitions).}

\begin{proposition}[see \cite{ulcigrai_absence_2011}]\label{thm:derivativeestimates}
{	Given any compact subset $K\subseteq \tilde{\mathcal M}$ such that there exists $D_K>0$ so that for any $T\in K$, $A_K(T)>0$ and $\Vert A_K(T)\Vert\leq D_K$, 
	there exists a measurable subset $E\subseteq K$ of positive Lebesgue 
	measure such that, if the iterates under normalized extended Rauzy-Veech 
	induction $( \mathcal{R}^{k}(\pi,\lambda,\tau))_{k\in \mathbb{N}}$ visit 
	$E$ along some increasing subsequence $(k_n)_{n\in\N} \subseteq \N$, i.e., 
	${\tilde{\mathcal{R}}}^{k_n}(\pi,\lambda,\tau)\in E$ for any $n\in \N$, 
	then, for any $\varphi\in LSS^2(\bigsqcup_{\alpha\in\mathcal A} I_\alpha)$ 
	there exists a constant $M=M(\varphi)$ such that $(k_n)_{n\in \N}$ are 
	$(M,\varphi)$-good derivative control times. }
\end{proposition}

\begin{proof}
{ Proposition~\ref{thm:derivativeestimates} follows from the proof of 
\cite[Proposition 4.2]{ulcigrai_absence_2011}, with the set $E$ being the 
measurable set $E_D$ in the proof of Proposition~\ref{thm:derivativeestimates}, 
as long as we can verify that it can be applied to $\varphi\in 
LSS^2(\bigsqcup_{\alpha\in\mathcal A} I_\alpha)$. Indeed, in 
\cite{ulcigrai_absence_2011}, the author was investigating roof functions of 
special flows, which are special representations of locally Hamiltonian flows 
{and therefore the result was proved for functions of the form 
\eqref{eq:logartihmic_singularities} under the additional assumption that, for 
every $\alpha\in\A$, the constants $C^+_{\alpha}$ and $C^-_\alpha$ were 
positive, while we now want to apply it to the more general class of cocycles 
investigated in this article, for which this is not necessarily the case}. 
However, one can see that the proof of Proposition 4.2 in 
\cite{ulcigrai_absence_2011} works verbatim for all functions $\varphi$ whose 
derivative $\varphi\varphi'$ has the form
\[
{\varphi'}(x)=g(x)+\sum_{\alpha\in\A} D_{\alpha}^+ 
\left(|I|\left\{\frac{(x-l_{\alpha})}{|I|} \right\}\right)^{-1} 
+\sum_{\alpha\in\A} D_{\alpha}^- \left(|I|\left\{\frac{({r_{\alpha}-x})}{|I|} 
\right\}\right)^{-1},
\]
where $\{ x \} $ denote the fractional part of $x\in [0,1]$, $g$ has bounded variation and the (possibly negative) constants $D_\alpha^\pm \in \R$ are such that } {
\begin{equation}\label{eq: derconst}
\sum_{\alpha\in\mathcal A}\left( D_{\alpha}^+ +D_{\alpha}^-\right)=0.
\end{equation}}
{ We claim that cocycles ${\varphi}\in 
LSS^2(\bigsqcup_{\alpha\in\mathcal A} I_\alpha)$, that we consider in this 
paper, have this form}. Indeed, if
\[
\varphi (x) = g_\varphi (x) + \sum_{\alpha \in \A} C_\alpha^+ \log\left(|I|\left\{\frac{(x - l_\alpha)}{|I|}\right\}\right) + C_\alpha^- \log\left(|I|\left\{\frac{(r_\alpha - x)}{|I|}\right\}\right),
\]
then 
\[
\varphi '(x) = g'_\varphi (x) + \sum_{\alpha \in \A} C_\alpha^+ \left(|I|\left\{\frac{(x - l_\alpha)}{|I|}\right\}\right)^{-1} - C_\alpha^- \left(|I|\left\{\frac{(r_\alpha - x)}{|I|}\right\}\right)^{-1}.
\]
Thus
\[
D_{\alpha}^+=C_{\alpha}^+\ \text{and}\ D_{\alpha}^-=-C_{\alpha}^-\text{ for every }\alpha\in\A.
\]
By the definition of $LSS^2(\bigsqcup_{\alpha\in\mathcal{A}} I_\alpha)$, we have $\sum_{\alpha\in\mathcal{A}}C_{\alpha}^+=\sum_{\alpha\in\mathcal{A}}C_{\alpha}^-$. 
Thus, \eqref{eq: derconst} is satisfied and we can apply {\cite[Proposition 4.2]{ulcigrai_absence_2011},} 
to derivatives of cocycles in $LSS^2(\bigsqcup_{\alpha\in\mathcal{A}} I_\alpha)$. 
\end{proof}
\begin{remark}
{ The proof of Proposition~\ref{thm:derivativeestimates} actually shows that, in our setting, 
\[
\sum_{\alpha\in\mathcal A}C_{\alpha}^+=\sum_{\alpha\in\mathcal A}C_{\alpha}^-=0.
\]}
\end{remark}

\subsubsection{Existence of good renormalization times}\label{sc:proof_full_measure_cond}
 We are now in a position to prove Proposition \ref{prop:full_measure_cond}, {which follows from by now classical arguments using (renormalized) Rauzy-Veech induction and exploiting its ergodicity.} 
\begin{proof}[Proof of Proposition \ref{prop:full_measure_cond}] { Let us show that there exists a compact subset} $K\subseteq \tilde{\mathcal M}$ such that every $(\pi,\lambda,\tau)\in K$ satisfies the following conditions:
	\begin{enumerate}[(a)]
		\item \label{cond:sym} $\pi$ is a symmetric permutation,
			\item \label{cond:drift} $(\pi,\lambda)$ has a $\delta$-drift,
			\item \label{cond:balanced} $\lambda$ is $\nu$-balanced,
				\item \label{cond:balanced height} $h$ is $\nu$-balanced for 
				some {$\nu<3$}, where $h=-\Omega_{\pi}\tau$, 
				{ and $\Omega_{\pi}$ was defined in 
				\eqref{eq:Omega},
\item \label{cond:tau2}$\tau$ is such that	
	$$
\sum_{\alpha\in\mathcal A}\tau_{\alpha}>0>\sum_{\alpha\in\mathcal 
A}\tau_{\alpha}-\tfrac{1}{2}\tau_{\pi_1^{-1}(d)},$$
		\item \label{cond:tau} for every $i\in\{2,\ldots,d-1\}$, we have $$\frac{ h_{\pi_0^{-1}(i)}}{2}<\sum_{j<i} \tau_{\pi_0^{-1}(j)}, $$ 
		}
		
		\item \label{cond:path} $K$ satisfies the assumptions of Proposition~\ref{thm:derivativeestimates};	
	\end{enumerate}
{ 
To build $K$, first choose $\lambda_0 $ as a positive vector such that $(\pi,\lambda_0)$ is Keane and has $\delta$-drift (which exists, if $\delta< \tfrac{1}{2d}$, in view of Remark~\ref{rk:notempty}). We claim that we can find a suspension datum $\tau$ such that, setting $h=-\Omega_\pi \tau$, condition \eqref{cond:tau2} as well as \eqref{cond:tau} hold. }
	{ For this purpose, take 
	$\tau_0:=\big(1,\frac{1}{100(d-2)},\ldots,\frac{1}{100(d-2)},-1\big)$. Then
	\[
	\sum_{\alpha\in\A}\tau_{\alpha}=\frac{1}{100}>0\quad\text{and}\quad
	\sum_{\alpha\in\A}\tau_{\alpha}- 
	\frac{1}{2}\tau_{\pi_1^{-1}(d)}=-\frac{49}{100}<0,
	\]
	thus it satisfies Condition \eqref{cond:tau2}. Moreover,
	\[
	h=-\Omega_{\pi}\tau_0=\left(\frac{99}{100}, 
	2-\frac{d-3}{100(d-2)},2-\frac{d-5}{100(d-2)},\ldots,2+\frac{d-5}{100(d-2)},
	2+\frac{d-3}{100(d-2)},\frac{101}{100}\right).
	\]
	Thus, for $i\in\{2,\ldots,d-1\}$, we get
	\[
	\frac{\sum_{j<i}\tau_{\pi_0^{-1}(j)}}{h_{\pi_0^{-1}}}
	=\frac{1+\frac{i-2}{100(d-2)}}{2-\frac{d-3}{100(d-2)}+\frac{2(i-2)}{100(d-2)}}
	=\frac{1}{2}+\frac{d-3}{2(200(d-2)-(d-3)+2(i-2))}>\frac{1}{2},
	\]
	and, in particular, \eqref{cond:tau} is satisfied.
	}
	{ 
	
Using the fact that Rauzy-Veech induction cocycle is locally continuous, we can find an open neighborhood $U$ of the triple $(\pi,\lambda_0,\tau_0)$ such that conditions \eqref{cond:sym}, \eqref{cond:drift} and \eqref{cond:tau} are satisfied, for any triple in $U$. Then any compact $K\subseteq U$ is such the length and height vectors of any triple in $K$ satisfy conditions \eqref{cond:balanced} and \eqref{cond:balanced height}. Finally, to guarantee that \eqref{cond:path} holds, we want to ensure that for every $(\pi,\lambda', \tau')\in K$ the accelerated Rauzy-Veech matrix $A_K=A_K(\pi, \lambda, \tau')$ is positive and $\Vert A_K(T')\Vert $ is bounded. This can be guaranteed by choosing $K\subseteq U$ of the form $K= \{\pi\} \times \Delta_\gamma\times V$, where $\gamma$ is a simple loop starting from $\pi$ and $V\subseteq \Theta_{\pi}$ open, see \cite{avila_exponential_2006} and \cite{ulcigrai_absence_2011} for defnitions and details.} 
	
Let $E\subseteq K$ be the subset given by Proposition \ref{thm:derivativeestimates}. 
{
Then, since $\tilde{\mathcal R}$ is ergodic and $E$ is of positive Lebesgue measure (and hence has positive measure for the natural invariant measure for $\tilde{\mathcal R}$), for a.e. IET $(\pi,\lambda),$ there exists $\tau\in \Theta_\pi$, {which satisfies \eqref{cond:balanced height}}, such that $(\pi,\lambda,\tau)$ is infinitely recurrent to $E\subseteq K$, i.e.,} there exists an increasing sequence $(k_n)_{n\in\N}$ such that $\tilde{\mathcal R}^{k_n}(\pi,\lambda,\tau)\in E$. We claim that $(k_n)_{n\in\N}$ is the sought sequence.
	
First note that by the choice of $K$ (and \eqref{eq: matrixbound}), $(k_n)_{n \in \N}$ is a sequence of $\nu$-balanced times and $\mathcal{RV}^{k_n}(T)$ has a $\delta$-drift. Moreover, by Theorem \ref{thm:derivativeestimates}, $(k_n)_{n \in \N}$ is also a sequence of good derivative controls for $T$. {Furthermore, in view of Remark~\ref{rem:onlyforward}, since the assumption \eqref{cond:tau2} ensures that $\tau^n$ satisfies \eqref{eq:tauchoice}, we can write each midpoint $c_\beta$, with $\beta\in \A\cup \{ \tfrac{1}{2}\}$, as $T^{s^n_\alpha} (c_\alpha)$, for some $\alpha \in \A\cup \{ \tfrac{1}{2}\}$ with $s^n_\alpha$ positive as in \eqref{eq:positive}.} 

It remains to see that $T$ has well-positioned discontinuities at time $k_n$ for every $n\in\N$. This follows from Condition \eqref{cond:tau} in the definition of $K$ and \eqref{cond:balanced height}. 
 Indeed, for every $\alpha\in\mathcal A\setminus\{\pi_0^{-1}(1),\pi_0^{-1}(d)\}$, by Condition \eqref{cond:tau}, the {rectangle with base} $I^{k_n}_{\alpha}$ { and height} $h^{k_n}_{\alpha}$ contains a singularity at height 
$$\sum_{j<\pi_0(\alpha)} \tau_{\pi_0^{-1}(j)}>h^{k_n}_{\alpha}/2$$ 
(see Chapter 15 in \cite{viana_ergodic_2006}). Moreover, the tower $\bigcup_{i=0}^{q^{k_n}_{\alpha}} I^{k_n}_{\alpha}$ consists of the intersections of the initial interval $I$ with the said {rectangle}. Since { the chosen $\tau$ satisfies \eqref{cond:balanced height}}, the lower half of the {rectangle} contains at most three times fewer levels of a tower over the interval $I^{k_n}_{\alpha}$ than the upper half (compare with the proof of Lemma \ref{lm: position_of_centers}), the desired condition follows.
\end{proof}

\subsection{Final arguments to show the skew-products ergodicity}\label{sec:last}
In this section, we conclude the proof of the main result of this article, namely Theorem \ref{thm:reductiontoIETs}.
{ Let us first explain the strategy. To prove ergodicity, we need to verify its assumptions of the ergodicity criterion given by Proposition \ref{prop:ergodicity_criterion} for the sets built in \S~{\ref{sec:goodtimes}}}. The sets are constructed to be partial rigidity sets, so we need to check the Tightness Assumption 1 and the Oscillation Assumption 2. 

\smallskip
To check Condition \eqref{cond:tightness} in Proposition \ref{prop:ergodicity_criterion}, i.e., to show that the Birkhoff sums on the set $\Xi_n$ that we consider are tight, we will use that, in view of the cancellations due to the {assumption that the cocycle is odd on each continuity interval}, there exist points where the Birkhoff sums are small (see Proposition~ {\ref{lem: zeroes}}) and then { we use that the good renormalization times (in view of their Definition~\ref{def:goodtimes}) have good derivative controls (see Definition~\ref{def:cancelltimes})} to deduce tightness. 

	\smallskip
To check Condition \eqref{cond:decay_coefficients} in Proposition \ref{prop:ergodicity_criterion}, in particular, to estimate the integral appearing in Condition \ref{cond:decay_coefficients} of Proposition \ref{prop:ergodicity_criterion}, we will use the following lemma:

\begin{lemma}
	\label{lem:decay_coefficients}
	Let $\psi:[0, 1] \subseteq \R \to \R$ be a $C^2$ function. Assume there exist $0 < \eta < 1$ and $V > 0$ such that 
	\[\rho = \min_{x \in [0, \eta]} |\psi'(x)| > 0, \quad \quad \textup{Var}(\psi')|^\eta_0 \leq V.\]
	Then, there exists a constant $C > 0$, depending only on $V$ and $\rho$, such that, { for any $k>1$,}
	\[ \left| \int_0^1 e^{2\pi k i \psi(x)} d\textup{Leb}(x) \right| \leq 1 - \eta + \frac{C}{k}. \]
\end{lemma}
 Let us first conclude the strategy overview and then give the proof of the lemma above.
To apply Lemma \ref{lem:decay_coefficients} to each floor of the tower $\Xi_n$, for $n\in\N$ sufficiently large, the idea is to exploit the control given by the $\delta$-drift (in particular the good base dynamics described by {Lemmas \ref{lem:rigid_interval} and \ref{lem:towers_approaching_left_singularities}}) in order to approach the left singularity in each of the towers to get the lower bound on the derivative (that is, to beat the constant coming from the good cancellation condition) while not approaching too close to remain uniformly bounded. 
	
\smallskip
Let us now prove the Lemma and then present the details of the proof of Theorem~\ref{thm:reductiontoIETs}.

\begin{proof}[Proof of Lemma \ref{lem:decay_coefficients}]
	Let us assume WLOG that $\psi'(x) \geq \rho$ for any $0 \leq x \leq \eta.$ { From the definition of variation, using that for any interval $(x_i,x_{i+1})\subseteq [0,\eta]$
	$$
	\left| \frac{1}{\psi(x_i)}- \frac{1}{\psi (x_{i+1})} \right| \leq \frac{\left| \psi (x_{i+1})- \psi (x_i) \right|}{\min_{x_i\leq x\leq x_{i+1}}|\psi'(x)|^2},
	$$}
one can show that
	\[\left. \textup{Var}(1/\psi')\right|_0^\eta \leq \frac{\textup{Var}(\psi')}{(\min_{x \in [0, \eta]} |\psi'(x)|)^2}\]
and therefore, 	since $\left. \psi'\right|_0^\eta$ is strictly positive, for $\rho$ and $V$ as in the assumptions of the Lemma, we have
$$\left|	\int_0^\eta \frac{d}{dx}\left( \frac{1}{\psi'(x)}{\textrm{d} \mathrm{Leb}(x)} \right) \right|\leq \left. \textup{Var}(1/\psi')\right|_0^\eta \leq \frac{V}{ \rho^2}.
	$$ 
Using this estimate and integrating by parts, we have
\begin{align*} 
	\left| \int_0^\eta e^{2\pi k i \psi(x)} \textrm{d} \textup{Leb}(x) \right| & = \left| \int_0^\eta \frac{\frac{d}{dx}\left(e^{2\pi k i \psi(x)}\right)}{2\pi k i \psi'(x)} \textrm{d} \textup{Leb}(x) \right| \\
	 & = \frac{1}{2\pi k} \left| \left[ \frac{e^{2\pi k i \psi(x)}}{\psi'(x)} \right]_0^\eta 
	- \int_0^\eta e^{{2\pi k i \psi(x)}}\frac{d}{dx} 
\left( \frac{1}{\psi'(x)} \right)	 \textrm{d} { \textup{Leb}(x)} 	\right|\\
 & {\leq \frac{2}{2\pi k\, {\min_{0\leq x\leq \eta}|\psi'(x)|^2}} + \frac{1}{2\pi k} \frac{V}{\rho^2}
 \leq \frac{1}{k} \frac{1}{\pi \rho } \left( 1+ \frac{V}{2\rho} 
	\right) }\end{align*} 
	Therefore, { for $C: = \frac{1}{\pi \rho } \left( 1+ {V}/{2\rho} 
	\right)$, } we have
	\[ \left| \int_0^1 e^{2\pi k i \psi(x)} d\textup{Leb}(x) \right| \leq 1 -\eta + \left| \int_0^\eta e^{2\pi k i \psi(x)} d\textup{Leb}(x) \right| \leq 1 - \eta + \frac{C}{k}, \]
which concludes the proof. \end{proof}

\begin{proof}[Proof of Theorem \ref{thm:reductiontoIETs}]
{
 Let us consider the full measure set of IETs given by Proposition \ref{prop:full_measure_cond} and assume that $T$ belongs to this set. 
Let $f$ be as in Theorem \ref{thm:reductiontoIETs}. Since $f$ has non trivial logarithmic singularities (see Definition~\ref{def:logsing}), there exists $\alpha \in \mathcal{A}$ and $1\leq k\leq d$ such that $\alpha= \pi_0^{-1}(k)$ and either}
\begin{align}\label{case: first_plus_last_tower}
& 1 < k < d \quad \text{ and } \quad C_{\pi_0^{-1}(k)}^+ \neq 0, \qquad \text{or} \\
& k = d \quad \text{ and } \quad C_{\pi_0^{-1}(1)}^+ + C_{\pi_0^{-1}(d)}^+ \neq 0. \label{case: middle tower}
\end{align}
\noindent {
Let $(m_n)_{n\in\N}$ be the sequence of $(\nu, \delta, M)$-good times given by Proposition \ref{prop:full_measure_cond}, for which also the second part of Lemma~\ref{lm: position_of_centers} (in particular \eqref{eq:positive}) holds.} 
Let $D=D(\nu, \delta, M)$ be the constant given by Proposition~\ref{lem:BScontrol}. 
{Choosing $\eta >0$ sufficiently small, we can find $\rho>0$ such that
\begin{equation}\label{def:rho} 0<\rho<
\begin{cases}
 \frac{d\nu|C^+_{\alpha}|}{2\eta}-D
 & \text{if}\ \eqref{case: first_plus_last_tower}\ \text{holds,\ or}
 \\ 
\frac{\left|C_{\pi_0^{-1}(1)}^+ + {C_{\pi_0^{-1}(d)}^+}\right|}{2\eta}-D& 
 \text{if \ \eqref{case: middle tower} \ holds.}
 \end{cases}
\end{equation}
Furthermore, choosing a smaller $\eta>0$ if necessary, we can also assume that 
$\eta<\delta$.}

For any $n \in \N$, fix $0<\eta<1/4$, { so that $\rho$ as in \eqref{def:rho} can be chosen}, and define $\Xi_n(\eta)$ and $h^n = q_\alpha^n + q_{\overline\alpha}^n$ as in \eqref{eq:heights_Xi_n}. As shown at the end of \S~\ref{sec:goodtimes}, from the definition of $\Xi_n$ and $h_n$ in \S~\ref{sec:goodtimes} and Lemma \ref{lem:towers_approaching_left_singularities}, it follows that $\{\Xi_n\}_{n\in\N}$ is a partial rigidity sequence {with partial rigidity times $\{ 2h_n\}_{n\in \N}$,} in the sense of Definition \ref{def:partial_rigidity}. We now proceed to verify the two conditions of the ergodicity criterion.

{
\medskip
\noindent {\it Tightness of Birkhoff sums.} To show tightness, we will compare the Birkhoff sums $S_{2h_n}\varphi(x)$ of any $x\in \Xi_n$ with the Birkhoff sums $S_{2h_n}\varphi(x_\alpha^n)$, { where the point $x^n_\alpha$ belongs to the tower over $I^{m_n}_\alpha$ given by Proposition~\ref{lem: zeroes}. Recall that, in particular, this explicit Birkhoff sum is either equal to 0 or it is very close to 0.}

To study the difference $|S_{2h_n}\varphi(x) - S_{2h_n}\varphi(x_\alpha^n)|$, we can split both the orbit segments of $x$ and $x_\alpha^n$ of length $2h_n$ into two orbit segments and pair them so that we can apply the mean value theorem to the differences of the corresponding Birkhoff sums as follows. 

Assume, for example, $k:=j-i\geq 0$ (the other case is similar). Since $x=T^i(x_0)$, we have that $T^k(x)=T^{j-i}(T^i(x_0))=T^j(x_0)$, so $T^{k}(x)$ and $x^n_\alpha=T^j(y)$, { where $y\in I_{\alpha}^{m_n}$}, belong to the same floor of the Rohlin tower over $F^n_\alpha$. Furthermore, by the good dynamics, if $J_1$ denotes the interval between them, $T^\ell $ restricted to $J_1$ is as an isometry (and in particular continuous) for $0\leq \ell < h_n-k$ (by the good dynamics given by Lemma~\ref{lem:towers_approaching_left_singularities}, see in particular \ref{cond:nested_images}). Hence, we can apply the mean value theorem so that, for some $z_1\in J_1$, 
\begin{equation}\label{eq:meanvalue1}
\left| S_{h_n-{k}}\varphi (T^k (x)) - S_{h_n-{k}}\varphi (x^n_\alpha)\right| \leq \left| S_{h_n-{k}}\varphi' (z_1) \right | |J_1|. \end{equation}
Similarly, we have that $T^{2h_n- k}(x^n_\alpha)$ and $x=T^i(x_0)$ belong to the same floor of the tower over $I^{m_n}_\alpha$ (since we can write $T^{2h_n- k}(x^n_\alpha) = T^{i-j} T^{2h_n} (T^j (y)) = T^i ( T^{2h_n} (y))$ where $y\in F^n_\alpha$ and, by the good dynamics, $T^{2h_n} (y)$ is again in $I^{m_n}_\alpha$). Thus, if $J_2$ denotes the interval between them, $T^\ell$ restricted to $J_2$ is an isometry for $0\leq \ell< k$. Thus, as above, by the mean value theorem, there exists $z_2\in J_2$ such that
\begin{equation}\label{eq:meanvalue2}
\left| S_{k} \varphi (x) - S_{{k}}\varphi (T^{2h_n- k} x^n_\alpha)\right| \leq 
\left| S_{{k}} \varphi' (z_2) \right | |J_2|.
\end{equation}
Let us now estimate the RHS of \eqref{eq:meanvalue1} and \eqref{eq:meanvalue2}. 
By the $M$-good derivative control (see Definition~\ref{def:cancelltimes}) and the control on the closest visit given by Conditions \ref{cond:3} or \ref{cond:4} of Lemma~\ref{lem:towers_approaching_left_singularities}, recalling that {we chose $\eta<\delta$}, we can estimate 
\begin{equation}\label{eq:der1}
\left| S_{{k}} \varphi' (z_2) \right | \leq M k + \frac{(2d-1)C_{max}}{\delta|I^{m_n}| } + \frac{C^+_{\gamma(\alpha)}}{\eta |I^{m_n}_\alpha|}\leq M q^{m_n}_\alpha + \frac{2d
C_{max}}{\eta|I^{m_n}|}.
\end{equation}
Similarly, writing $z_1 =T^k(z_0)$ for $z_0\in F^n_\alpha$ and seeing the RHS of \eqref{eq:meanvalue1} as a difference, using Proposition~\ref{lem:BScontrol} and an estimate similar to the above, we get
\begin{equation}\label{eq:der2}
 \left| S_{2h_n-k} \varphi' (z_1)\right| = \left| S_{2h_n} \varphi' (z_0) - 
 S_{k} \varphi' (z_0)\right| \leq ( D+M) q_{\alpha}^{m_n} + \frac{(2d+1) 
 C_{max}}{\eta|I^{m_n}| } .
\end{equation}
Combining all these estimates (i.e., adding up \eqref{eq:meanvalue1} and \eqref{eq:meanvalue2} and applying \eqref{eq:der1} and \eqref{eq:der2}), using balance (which gives $1/|I^{m_n}|\leq d\nu q^{m_n}_\alpha$) and recalling that $|J_i|\leq |I^{m_n}_\alpha|$, for $i=1,2$, we get
\begin{align*}
\big| S_{2h_n}\varphi(x) - S_{2h_n}\varphi(x_\alpha^n)\big| & \left| \left(S_{k}\varphi(x) + S_{2h_n-k}\varphi(T^k x) \right)- \left( S_{2h_n-k}\varphi(x_\alpha^n)+ S_{k}\varphi(T^{2h_n-k }x_\alpha^n) 
\right)\right|
\\ & \leq \left| S_{k} \varphi (x) - S_{{k}}\varphi (T^{2h_n- k} x^n_\alpha)\right| +
\left| S_{2h_n-{k}}\varphi (T^k x) - S_{2h_n-{k}}\varphi (x^n_\alpha)\right| \\
& \leq \overline{D} q^{m_n}_\alpha 
|I^{m_n}_\alpha|
\leq \overline{D}, \quad \text{where} \ \overline{D}:=D+2M+ \frac{\nu (2d^2+d) 
C_{max}}{\eta},
\end{align*}
for any $n\in\N$ and $x\in \Xi_n$. }This, together with \eqref{eq" known_bounded_sum}, shows that { 
the sets $\{\Xi_n\}_{n\in \N}$ satisfy} Condition \ref{cond:tightness} in 
Proposition \ref{prop:ergodicity_criterion}.

\medskip
{
\noindent {\it Oscillations of Birkhoff sums.} We now proceed to show that Condition \ref{cond:decay_coefficients} in Proposition \ref{prop:ergodicity_criterion} is also satisfied. 
The proof differs slightly depending on whether the assumption \eqref{case: first_plus_last_tower} on the constants or the assumption \eqref{case: middle tower} is satisfied. However, the only real difference comes from the fact that in the second case, namely \eqref{case: first_plus_last_tower} (when only the sum of two constants is known to be non-zero), we approach two discontinuities of $f$ at once (since there is one in 0 and one in the pre-image of 0 via $T$), while in the case \eqref{case: middle tower} we approach only one singularity. 
}

	We will give the argument only in the case \eqref{case: middle tower}, leaving the details of the other case to the reader. In this case, 
let $\widetilde{F}^{n}_{\alpha}$ be the left-hand side initial interval of $F^{n}_{\alpha}$ (namely, an interval with the same left endpoint) of length $\eta\cdot|F^{n}_{\alpha}|\le \eta/{q_{\alpha}^{{m_n}}}$, and consider {the subtower of $\Xi_n$ above it (namely, a Rohlin tower of height $1/4$th of the whole tower), i.e.,} 
$$\widetilde{\Xi}_n:=\bigcup_{i=0}^{\lceil q^{{m_n}}_{\alpha}/4\rceil}T^i\big(\widetilde{F}^n_{\alpha}\big)\subseteq \Xi_n.$$
Note that each of the intervals $T^i\big(\widetilde{F}^n_{\alpha}\big)$ is again { the left subiterval of $T^i\big(F^n_{\alpha}\big)$. Recalling that the left endpoint of $F^{n}_{\alpha}$ (and hence $\widetilde{F}^{n}_{\alpha}$) is $l^{m_n}_\alpha + \eta |I^{m_n}_\alpha|$ and $\widetilde{F}^{n}_{\alpha}$ has length $\eta| I^{m_b}_\alpha|$, for any $x\in\tilde \Xi_n$, the closest visit to a singularity (given by~\eqref{eq:closest}), as discussed in \S~\ref{sec:estimates}) is not closer than $2\eta|I^{m_n}_\alpha|$. Thus, Proposition~\ref{lem:BScontrol} and balance (which guarantees that $|I^{m_n}| q_\alpha^{m_n} \geq \min_{\beta\in \A} \left(|I^{m_n}_\beta| q_\beta^{m_n} \right) \geq 1/d\nu$) give
\begin{equation}\label{eq: lower bound on derivative}
|S_{2h_n}{\varphi'}(x)|\ge 
\frac{|C^+_{\gamma(\alpha)}|}{|x_0-l^{m_n}_\alpha|} -D q_{\alpha}^{m_n} \geq
\frac{|C^+_{\gamma(\alpha)}|}{2\eta |I^{m_n}|}
-D\, q_{\alpha}^{m_n} > \left(\frac{d\nu|C^+_{\gamma(\alpha)}|}{2\eta}-D\right) q_{\alpha}^{m_n} \geq \rho \, q_{\alpha}^{m_n} , 
\end{equation}
where the last inequality follows by choice of $\eta$ and definition of $\rho$, see \eqref{def:rho}.} 
On the other hand, Proposition~\ref{lem:BScontrol}, Lemma \ref{lem:towers_approaching_left_singularities} and balance give
\begin{equation}\label{eq: upper bound on derivative}
|S_{2h_n}{\varphi'}(x)|
\le Dq^{m_n}_{\alpha} + \frac{|C^+_{\alpha}|q^{m_n}_{\alpha}}{|x_0-l^{m_n}_\alpha| } 
\le Dq^{m_n}_{\alpha} + \frac{|C^+_{\alpha}|}{\eta |I^{m_n}_{\alpha}|} \leq 
\left( D+ \frac{d\nu |C^+_{\alpha}|}{\eta}\right) q^{m_n}_{\alpha}.
\end{equation}
 {The inequalities \eqref{eq: lower bound on derivative} and \eqref{eq: upper bound on derivative} imply that for every $n\in\N$ and every $i\in\{0,\lceil q_\alpha^{m_n}/4 \rceil \}$, the function $\psi:[0,1]\to[0,1]$, $\psi(x):=S_{2h_n}{\varphi}(|F_\alpha^n|x + T^i l_{\alpha}^{m_n})$, satisfies the assumptions of Lemma \ref{lem:decay_coefficients}. Indeed, by \eqref{eq: lower bound on derivative}, for every $x\in [0,\eta]$, 
 \[
\psi'(x)=\frac{d}{dx}S_{2h_n}{\varphi}(|F_\alpha^n|x + T^i 
l_{\alpha}^{m_n})\ge |F_{\alpha}^n|\rho q_{\alpha}^{m_n}\ge \rho/ d\nu,
\]
{ where the last inequality follows from the balance.}
Moreover, by \eqref{eq: upper bound on derivative}, we get
\[
Var(\psi'|_{[0,\eta]})=\int_0^\eta |\psi'(x)|\,dx\le 
|F_{\alpha}^n|\left(D\eta+d\nu|C_{\alpha}^+|\right)q_{\alpha}^{m_n}\le 
\left(D\eta+d\nu|C_{\alpha}^+|\right).
\]}
Thus, by Lemma \ref{lem:decay_coefficients}, there exists a constant $C>0$, which does not depend on $n$, such that, for every $0\le i\le {q_{\alpha}^n }$ and $k{>} 0$,

\[
\left|\int_{T^iF^n_{\alpha}}e^{2\pi ki\varphi(x)}\,d\textup{Leb}(x)\right|\le 
Leb(F^n_{\alpha})\left(1-\eta+\frac{C}{k}\right).
\]
Hence, for $k$ sufficiently large,
\[
\left|\int_{\Xi_n}e^{2\pi ki\varphi(x)}\,d\textup{Leb}(x)\right|< 
\textup{Leb}(\Xi_n),
\]
which shows that the sets $\{\Xi_n\}_{n \in \N}$ satisfy also Condition~\ref{cond:decay_coefficients} in Proposition \ref{prop:ergodicity_criterion}.
\smallskip

{
Thus, having verified all assumptions of Proposition \ref{prop:ergodicity_criterion}, we can apply it to conclude that the skew-product $T_\varphi $ under consideration is ergodic.}
\end{proof}

\subsubsection*{Acknowledgements} We thank Krzysztof Fr{\k a}czek for valuable discussions. The authors acknowledge the support of the {\it Swiss National Science Foundation} through Grant $200021\_188617/1$. The first author also thanks the {\it National Science Centre (Poland)} grant OPUS 2022/45/B/ST1/00179.

\bibliographystyle{acm}
\bibliography{IETinfinite.bib, Bibliography.bib, IET.bib, ExtraBib.bib}

\begin{thebibliography}{10}

\bibitem{aaronson_introduction_1997}
{\sc Aaronson, J.}
\newblock {\em An introduction to infinite ergodic theory}, vol.~50 of {\em
  Mathematical {Surveys} and {Monographs}}.
\newblock American Mathematical Society, Providence, RI, 1997.

\bibitem{aaronson2016discrepancy}
{\sc Aaronson, J., Bromberg, M., and Nakada, H.}
\newblock Discrepancy skew products and affine random walks.
\newblock {\em arXiv preprint arXiv:1603.07233\/} (2016).

\bibitem{aaronson_visitors_1982}
{\sc Aaronson, J., and Keane, M.}
\newblock The {Visitors} to {Zero} of {Some} {Deterministic} {Random} {Walks}.
\newblock {\em Proceedings of the London Mathematical Society s3-44}, 3 (May
  1982), 535--553.

\bibitem{arnold_topological_1991}
{\sc Arnol'd, V.~I.}
\newblock Topological and ergodic properties of closed 1-forms with
  incommensurable periods.
\newblock {\em Functional Analysis and its Applications 25}, 2 (1991), 81--90.

\bibitem{atkinson}
{\sc Atkinson, G.}
\newblock Recurrence of co-cycles and random walks.
\newblock {\em J. London Math. Soc. (2) 13}, 3 (1976), 486--488.

\bibitem{avila_visits_2015}
{\sc Avila, A., Dolgopyat, D., Duryev, E., and Sarig, O.}
\newblock The visits to zero of a random walk driven by an irrational rotation.
\newblock {\em Israel Journal of Mathematics 207}, 2 (Apr. 2015), 653--717.

\bibitem{avila_exponential_2006}
{\sc Avila, A., Gou{\"e}zel, S., and Yoccoz, J.-C.}
\newblock Exponential mixing for the {Teichm{\"u}ller} flow.
\newblock {\em Publications Math{\'e}matiques de l'Institut des Hautes
  {\'E}tudes Scientifiques 104}, 1 (Nov. 2006), 143--211.

\bibitem{chaika_singularity_2021}
{\sc Chaika, J., Fr{\k a}czek, K., Kanigowski, A., and Ulcigrai, C.}
\newblock Singularity of the {Spectrum} for {Smooth} {Area}-{Preserving}
  {Flows} in {Genus} {Two} and {Translation} {Surfaces} {Well} {Approximated}
  by {Cylinders}.
\newblock {\em Communications in Mathematical Physics 381}, 3 (Feb. 2021),
  1369--1407.

\bibitem{Ch-R}
{\sc Chaika, J., and Robertson, D.}
\newblock Ergodicity of skew products over linearly recurrent {IET}s.
\newblock {\em J. Lond. Math. Soc. (2) 100}, 1 (2019), 223--248.

\bibitem{Co-Fr}
{\sc Conze, J.-P., and Fr\c{a}czek, K.}
\newblock Cocycles over interval exchange transformations and multivalued
  {H}amiltonian flows.
\newblock {\em Adv. Math. 226}, 5 (2011), 4373--4428.

\bibitem{conze_ergodicite_1976}
{\sc Conze, J.~P., and Keane, M.}
\newblock Ergodicité d'un flot cylindrique.
\newblock {\em Publications mathématiques et informatique de Rennes}, 2
  (1976), 1--7.

\bibitem{FL:ont}
{\sc Fayad, B., and Lema\'{n}czyk, M.}
\newblock On the ergodicity of cylindrical transformations given by the
  logarithm.
\newblock {\em Mosc. Math. J. 6}, 4 (2006), 657--672, 771--772.

\bibitem{Fr-Hu}
{\sc Fr{\k a}czek, K., and Hubert, P.}
\newblock Recurrence and non-ergodicity in generalized wind-tree models.
\newblock {\em Math. Nachr. 291}, 11-12 (2018), 1686--1711.

\bibitem{Fr-Le}
{\sc Fr{\k a}czek, K., and Lema\'{n}czyk, M.}
\newblock On symmetric logarithm and some old examples in smooth ergodic
  theory.
\newblock {\em Fund. Math. 180}, 3 (2003), 241--255.

\bibitem{fraczek_ergodic_2012}
{\sc Fr{\k a}czek, K., and Ulcigrai, C.}
\newblock Ergodic properties of infinite extensions of area-preserving flows.
\newblock {\em Mathematische Annalen 354}, 4 (Dec. 2012), 1289--1367.

\bibitem{Fr-Ul_Ehr}
{\sc Fr{\k a}czek, K., and Ulcigrai, C.}
\newblock Non-ergodic {$\Bbb{Z}$}-periodic billiards and infinite translation
  surfaces.
\newblock {\em Invent. Math. 197}, 2 (2014), 241--298.

\bibitem{Fr-Ul_BS}
{\sc Fr{\k a}czek, K., and Ulcigrai, C.}
\newblock On the asymptotic growth of birkhoff integrals for locally
  hamiltonian flows and ergodicity of their extensions.
\newblock {\em preprint arXiv:2112.05939\/} (2021).

\bibitem{fraczek_asymptotic_2021}
{\sc Frączek, K., and Ulcigrai, C.}
\newblock On the asymptotic growth of {Birkhoff} integrals for locally
  {Hamiltonian} flows and ergodicity of their extensions, Dec. 2021.
\newblock arXiv:2112.05939 [math-ph].

\bibitem{Ho}
{\sc Hooper, W.~P.}
\newblock The invariant measures of some infinite interval exchange maps.
\newblock {\em Geom. Topol. 19}, 4 (2015), 1895--2038.

\bibitem{Hu-We}
{\sc Hubert, P., and Weiss, B.}
\newblock Ergodicity for infinite periodic translation surfaces.
\newblock {\em Compos. Math. 149}, 8 (2013), 1364--1380.

\bibitem{keane_interval_1975}
{\sc Keane, M.}
\newblock Interval exchange transformations.
\newblock {\em Mathematische Zeitschrift 141}, 1 (Feb. 1975), 25--31.

\bibitem{kontsevich_connected_2003}
{\sc Kontsevich, M., and Zorich, A.}
\newblock Connected components of the moduli spaces of {Abelian} differentials
  with prescribed singularities.
\newblock {\em Inventiones mathematicae 153}, 3 (Sept. 2003), 631--678.

\bibitem{Kr}
{\sc Krygin, A.~B.}
\newblock Ergodicity of a transformation of the plane.
\newblock {\em Trudy Moskov. Orden. Lenin. \`Energet. Inst.}, 499 (1980),
  15--22.

\bibitem{lemanczyk_ergodic_1996}
{\sc Lema{\'n}czyk, M., Parreau, F., and Voln{\'y}, D.}
\newblock Ergodic properties of real cocycles and pseudo-homogeneous {Banach}
  spaces.
\newblock {\em Transactions of the American Mathematical Society 348}, 12
  (1996), 4919--4938.

\bibitem{marmi_cohomological_2005}
{\sc Marmi, S., Moussa, P., and Yoccoz, J.-C.}
\newblock The cohomological equation for {Roth}-type interval exchange maps.
\newblock {\em Journal of the American Mathematical Society 18}, 4 (2005),
  823--872.

\bibitem{masur_interval_1982}
{\sc Masur, H.}
\newblock Interval {Exchange} {Transformations} and {Measured} {Foliations}.
\newblock {\em Annals of Mathematics 115}, 1 (1982), 169--200.
\newblock Publisher: Annals of Mathematics.

\bibitem{oren_ergodicity_1983}
{\sc Oren, I.}
\newblock Ergodicity of cylinder flows arising from irregularities of
  distribution.
\newblock {\em Israel Journal of Mathematics 44}, 2 (June 1983), 127--138.

\bibitem{poincare_memoire_1881}
{\sc Poincaré, H.}
\newblock Mémoire sur les courbes définies par une équation différentielle
  ({I}).
\newblock {\em Journal de Mathématiques Pures et Appliquées 7\/} (1881),
  375--422.

\bibitem{poincare_sur_1885}
{\sc Poincaré, H.}
\newblock Sur les courbes définies par les équations différentielles
  ({III}).
\newblock {\em Journal de Mathématiques Pures et Appliquées 1\/} (1885),
  167--244.

\bibitem{Ra-Tr1}
{\sc Ralston, D., and Troubetzkoy, S.}
\newblock Ergodicity of certain cocycles over certain interval exchanges.
\newblock {\em Discrete Contin. Dyn. Syst. 33}, 6 (2013), 2523--2529.

\bibitem{Ra-Tr2}
{\sc Ralston, D., and Troubetzkoy, S.}
\newblock Residual generic ergodicity of periodic group extensions over
  translation surfaces.
\newblock {\em Geom. Dedicata 187\/} (2017), 219--239.

\bibitem{Rav:mix}
{\sc Ravotti, D.}
\newblock Parabolic perturbations of unipotent flows on compact quotients of
  {${\rm SL}(3,\Bbb{R})$}.
\newblock {\em Comm. Math. Phys. 371}, 1 (2019), 331--351.

\bibitem{schmidt_cocycles_1977}
{\sc Schmidt, K.}
\newblock {\em Cocycles on ergodic transformation groups}.
\newblock Macmillan {Lectures} in {Mathematics}, {Vol}. 1. Macmillan Co. of
  India, Ltd., Delhi, 1977.

\bibitem{schmidt_cylinder_1978}
{\sc Schmidt, K.}
\newblock A cylinder flow arising from irregularity of distribution.
\newblock {\em Compositio Mathematica 36}, 3 (1978), 225--232.

\bibitem{Sch:mix}
{\sc Schmidt, K.}
\newblock Dispersing cocycles and mixing flows under functions.
\newblock {\em Fund. Math. 173}, 2 (2002), 191--199.

\bibitem{Ul:ICM}
{\sc Ulcigrai, C.}
\newblock Dynamics and 'arithmetics' of higher genus surface flows.
\newblock {\em arXiv:2207.06404\/} (1922), 388--395.

\bibitem{Ulcigrai_mixing_2007}
{\sc Ulcigrai, C.}
\newblock Mixing of asymmetric logarithmic suspension flows over interval
  exchange transformations.
\newblock {\em Ergodic Theory and Dynamical Systems 27}, 3 (June 2007),
  991--1035.

\bibitem{ulcigrai_absence_2011}
{\sc Ulcigrai, C.}
\newblock Absence of mixing in area-preserving flows on surfaces.
\newblock {\em Annals of Mathematics. Second Series 173}, 3 (2011), 1743--1778.

\bibitem{veech_gauss_1982}
{\sc Veech, W.~A.}
\newblock Gauss {Measures} for {Transformations} on the {Space} of {Interval}
  {Exchange} {Maps}.
\newblock {\em Annals of Mathematics 115}, 2 (1982), 201--242.
\newblock Publisher: Annals of Mathematics.

\bibitem{viana_ergodic_2006}
{\sc Viana, M.}
\newblock Ergodic {Theory} of {Interval} {Exchange} {Maps}.
\newblock {\em Revista Matem{\'a}tica Complutense 19}, 1 (Apr. 2006), 7--100.
\newblock Number: 1.

\bibitem{viana_lectures_2014}
{\sc Viana, M.}
\newblock {\em Lectures on {Lyapunov} {Exponents}}.
\newblock Cambridge {Studies} in {Advanced} {Mathematics}. Cambridge University
  Press, Cambridge, 2014.

\bibitem{yoccoz_interval_2010}
{\sc Yoccoz, J.-C.}
\newblock Interval exchange maps and translation surfaces.
\newblock In {\em Homogeneous flows, moduli spaces and arithmetic}, vol.~10 of
  {\em Clay {Math}. {Proc}.} Amer. Math. Soc., Providence, RI, 2010, pp.~1--69.

\bibitem{Zo:how}
{\sc Zorich, A.}
\newblock How do the leaves of a closed {$1$}-form wind around a surface?
\newblock In {\em Pseudoperiodic topology}, vol.~197 of {\em Amer. Math. Soc.
  Transl. Ser. 2}. Amer. Math. Soc., Providence, RI, 1999, pp.~135--178.

\end{thebibliography}

\vspace{5mm}

P.~Berk, {Faculty of Mathematics and Computer Science, Nicolaus
	Copernicus University, ul. Chopina 12/18, 87-100 Toru\'n, Poland}\\
{zimowy@mat.umk.pl}
\vspace{3mm}

F.~Trujillo, {Institut für Mathematik, Universität Zürich, Winterthurerstrasse 
190, CH-8057 Zürich, Switzerland}\\
{frank.trujillo@math.uzh.ch }

\vspace{3mm}
C.~Ulcigrai, {Institut für Mathematik, Universität Zürich, Winterthurerstrasse 
190, CH-8057 Zürich, Switzerland}\\
{corinna.ulcigrai@math.uzh.ch}
\end{document}